\documentclass{amsart}
\usepackage{amssymb, latexsym}

\author{Tristan L\'{e}ger}
\address{Courant Institute of Mathematical Sciences, 251 Mercer Street, New York, NY 10012, USA }
\email{tleger@cims.nyu.edu}

\theoremstyle{plain}
\newtheorem{theorem}{Theorem}

\newtheorem{remark}[theorem]{Remark}
\newtheorem{proposition}[theorem]{Proposition}
\newtheorem{lemma}[theorem]{Lemma}
\newtheorem{corollary}[theorem]{Corollary}

\numberwithin{equation}{section} \numberwithin{theorem}{section}

\begin{document}

\title[Scattering for a particle interacting with a Bose gas]
{Scattering for a particle interacting with a Bose gas}

\vspace{-0.3in}
\begin{abstract}
We study the asymptotic behavior of solutions to an ODE - Schr\"{o}dinger type system that models the interaction of a particle with a Bose gas. We show that the particle has a ballistic trajectory asymptotically, and that the wave function describing the Bose gas converges to a soliton in $L^{\infty}.$
\end{abstract}

\maketitle

\tableofcontents

\section{Introduction}
\subsection{Background and results obtained}
In this paper we study a system of equations that models the interaction of a particle with a Bose gas. It was introduced by J. Fr\"{o}hlich, Z. Gang and A. Soffer in \cite{Fro3}, \cite{Fro4}. More precisely it is shown by these authors that in the mean-field limit, the dynamics of the interaction is described by the following system: 
\begin{equation} 
\begin{cases}
\dot{X}(t) &= P(t) \\
\dot{P}(t) &= g \int \nabla_x W (X(t)-x) \bigg \lbrace \vert \alpha(t) \vert ^2 - \frac{\rho_0}{g^2} \bigg \rbrace  dx \\
i \partial_t \alpha &= \Bigg( - \Delta + g W(X(t) - x) \Bigg) \alpha(t) + \kappa \Bigg( \phi \star \big \lbrace \vert \alpha(t) \vert^2 - \frac{\rho_0}{g^2} \big \rbrace \Bigg) \alpha
\end{cases}
\end{equation}
with $X : \mathbb{R} \rightarrow \mathbb{R}^3 , P : \mathbb{R} \rightarrow \mathbb{R}^3 , \alpha : \mathbb{R} \times \mathbb{R}^3 \rightarrow \mathbb{C}. $ The function $X$ denotes the position of the particle, and the wave function $\alpha$ is used to describe the state of the Bose gas. \\
$\phi: \mathbb{R}^3 \rightarrow \mathbb{R}$ and $W:\mathbb{R}^3 \rightarrow \mathbb{R}$ are interaction potentials between the particle and the field. \\
$g,\kappa$ are coupling constants and $\rho_0$ is a constant that quantifies the friction. The dot stands for time differentiation. \\
These equations are Hamiltonian with respect to
\begin{align*}
\mathcal{H} &= \frac{P^2}{2} + \int_{\mathbb{R}^3} \vert \nabla \alpha (x) \vert^2 + g W(X-x) \big( \vert \alpha (x) \vert^2 - \frac{\rho_0}{g^2} \big) dx  \\
&+ \frac{\kappa}{2} \int_{\mathbb{R}^6} \big( \vert \alpha (x) \vert^2 - \frac{\rho_0}{g^2} \big) \phi(y-x) \big( \vert \alpha (x) \vert^2 - \frac{\rho_0}{g^2} \big)~ dx dy.
\end{align*}
The boundary condition $\alpha (x) \longrightarrow \sqrt{\frac{\rho_0}{g^2}}$ as $\vert x \vert \rightarrow + \infty$ is imposed for the system to have finite energy. \\
It is then natural to introduce the new field variable $\beta(x) = \alpha(x) - \sqrt{\frac{\rho_0}{g^2}} $ with $\beta (x) \longrightarrow 0$ as $\vert x \vert \rightarrow +\infty.$ Writing the system using this unknown yields
\begin{equation} \label{full}
\begin{cases}
\dot{X} (t) &= P(t) \\
\dot{P} (t) &= g \int \nabla_x W (X(t)-x) \bigg \lbrace \vert \beta(t) \vert ^2 +2 \displaystyle \sqrt{\frac{\rho_0}{g^2}} \Re \beta (t) \bigg \rbrace  dx \\
i \partial_t \beta &= \Bigg( -\Delta + g W(X(t) - x) \Bigg) \beta(t) + \sqrt{\rho_0} W(X(t)-x) \\
&+ \kappa \Bigg( \phi \star \big \lbrace \vert \beta \vert^2 + 2 \sqrt{\frac{\rho_0}{g^2}} \Re \beta \big \rbrace \Bigg)(x) \Bigg( \beta + \sqrt{\frac{\rho_0}{g^2}} \Bigg).
\end{cases}
\end{equation}
This system presents several difficulties from the point of view of the description of the asymptotic behavior of solutions. Indeed the field satisfies a Schr\"{o}dinger type equation with a quadratic nonlinearity, which is known to be delicate to analyse. The method of space-time resonances has been developed by P. Germain, N. Masmoudi and J. Shatah (see for example \cite{GMS}) to tackle this kind of difficulty. Another related method has been developed by S. Gustafson, K. Nakanishi and T.-P. Tsai in \cite{GNT} to study the Gross-Pitaevskii equation.
\\
An additional difficulty for the system \eqref{full} is to understand the coupling between the two equations. This is the object of the present paper. Since this is already a significant obstacle, we focus on a simplified version of \eqref{full}, namely its Bogoliubov limit: the ratio $2 \kappa \rho_0 / g^2 := \lambda $ is kept constant (we will normalize it to 1) while $g$ and $\kappa$ are sent to 0. We are left with the following equations:
\begin{equation} \label{system}
\begin{cases}
i \partial _t \beta &= - \Delta \beta + \Re \beta + \sqrt{{\rho}_0} W(X(t)-x) \\
\dot{X} (t) &= P(t)  \\
\dot{P} (t) &=  \sqrt{{\rho}_0} \Re \langle \nabla _x W^{X(t)} , \beta \rangle 
\end{cases}
\end{equation}
where $\langle f , g \rangle : = \Re \int \bar{f} g$ denotes the inner product and $W^{X(t)}(\cdot):= W(X(t)-\cdot). $ \\
This system retains the coupling of the equations, but there is no nonlinear interaction of the field with itself. \\
Moreover it has an associated Hamiltonian:
\begin{align}
\label{hamiltonian} \mathcal{H}(X, P, \beta, \nabla \beta) = \frac{P^2}{2} + \int_{\mathbb{R}^3} \vert \nabla \beta \vert^2 dx + \int_{\mathbb{R}^3} \vert \Re \beta \vert^2 dx + 2 \sqrt{\rho _0} \int_{\mathbb{R}^3} W^{X} \Re \beta dx
\end{align}
Note that we will keep the parameter $\rho_0$ in the equations. We elected to normalize the potential in some sense rather than this physical parameter. It allows us to state our smallness assumptions in terms of this physical constant (which quantifies the friction in the system) rather than in terms of a less meaningful quantity related to the interaction potential. See the notation section of the paper for the normalization alluded to here. 
\\
The system \eqref{system} has already been studied by D.Egli, J. Fr\"{o}hlich, Z.Gang, A. Shao in \cite{E}. They proved global existence of solutions in some natural Sobolev spaces (their result includes global well-posedness in the energy space). They also identified two possible regimes for the particle: subsonic or supersonic, depending on whether the speed of the particle is above or below a certain threshold called speed of sound (it corresponds to the parameter $\lambda$ introduced earlier). The main reason for distinguishing between these two regimes is that there exist traveling wave solutions to \eqref{system} in the subsonic regime, but not in the supersonic regime. \\
In \cite{Fro1}, J. Fr\"ohlich and Z.Gang studied the asymptotic behavior of solutions. They proved that if the initial velocity of the particle and the initial data of the field are small enough, then the particle remains in the subsonic regime for all times and approaches ballistic motion asymptotically. They also proved that the wave function of the Bose gas has (asymptotically) the shape of a soliton. \\
These two same authors also studied the supersonic case in \cite{Fro2}, where a different phenomenon takes place: the supersonic particle decelerates by emission of what is known in the Physics literature as Cherenkov radiation (see for example \cite{J}), and approaches the speed of sound asymptotically. They assumed that the initial data for the field is small and that the initial velocity of the particle is within a well-calibrated supersonic range: it cannot be too close to the threshold between regimes, or too large. This restriction comes from the perturbative nature of their proof. They also suppose that a Fermi golden rule is satisfied by the potential of interaction. More precisely, they require its Fourier transform to vanish at a certain order at 0. This kind of condition appears naturally when dealing with systems of interaction between a particle and a field (see for example \cite{SWlin}). The method used to carry out the analysis is related to the technique of analytic deformations that was developed to study quantum resonances. 
\\
Let's also mention some works on related systems: the case where $\kappa$ and $g$ are taken such that $\kappa = 0$ and $g \to 0 $ has been treated by J. Fr\"{o}hlich, Z. Gang and A. Soffer in \cite{Fro3}. In the case where $\kappa =0$ and $g \neq 0$, D. Egli and Z. Gang describe in \cite{E2} the asymptotic behavior of the particle and the wave function of the gas. We also mention the work of A. Komech and E. Kopylova \cite{KK} on a closely related model, where the same type of result is obtained for subsonic initial velocity. It is also assumed that the initial data of both the initial velocity and the field are close to a traveling wave solution. The point of view and methods used in this paper are closer to what is classically used to study asymptotic stability of travelling waves of the nonlinear Schr\"{o}dinger equation (see for example \cite{BP}, \cite{Cucc}).
\\
\\
In this paper our goal is to study the asymptotic behavior of the system described by the equations \eqref{system} for arbitrary initial velocity of the particle in space dimension five. We will say more about the relevance of this high dimension assumption in Section \ref{disc} below. \\ 
We assume that the potential (smooth and localized) satisfies a kind of Fermi Golden rule, namely that its Fourier transform vanishes at the origin. More precisely the potentials considered will be of the form $W=(-\Delta)^n V.$ Under some smallness conditions on the initial data of the field and the potential we prove that, for arbitrary initial velocity, the particle has ballistic motion asymptotically and that the field has the shape of a soliton. We give below a somewhat imprecise version of our main theorem. For a complete statement, see Theorem \ref{main} in Section \ref{results} below.
\begin{theorem} \label{mainheur}
Let $n > 1/4$ be a real number. Let $W$ be a smooth, fast decaying function such that $W = (-\Delta)^n V.$ \\
Let $\beta_0 \in H^2 (\mathbb{R}^5) \cap \textbf{W}^{3,1}(\mathbb{R}^5).$ \\
If $\rho_0,$ $\beta_0$ and $W$ are small enough, then there exist $P_{\infty} \in \mathbb{R}^5$ and $\beta_{\infty} \in L^{\infty}(\mathbb{R}^5)$ such that
\begin{eqnarray*}
P(t) \longrightarrow_{t \rightarrow + \infty} P_{\infty} \\
\Vert \beta(t) - \beta_{\infty}(\cdot - X(t)) \Vert_{{\infty}} \longrightarrow_{t \rightarrow + \infty} 0
\end{eqnarray*}
where X(t) denotes the position of the particle. \\
Moreover $\beta_{\infty}$ solves the following linear elliptic equation:
\begin{displaymath} \left( \begin{array}{cc}
P_{\infty} \cdot \nabla & - \Delta \\
\Delta -1 & P_{\infty} \cdot \nabla
\end{array} \right) \beta_{\infty} = \rho_0 
\left( \begin{array}{ll}
0 \\
W
\end{array} \right)
\end{displaymath}
\end{theorem}
Here $H^s(\mathbb{R}^5)$ and $\textbf{W}^{s,p}(\mathbb{R}^5)$ denote the usual Sobolev spaces.
\begin{remark}
Notice that the Fermi-Golden rule in the statement is weaker than in \cite{Fro2}. In fact we could assume $n>0$ only, and the proof would hold with minor changes (see Remark \ref{remarquefermi}). We made the choice to prove the theorem for $n>1/4$ to avoid technical fractional integration by parts. 
\end{remark}

\subsection{Difficulties of the problem}
In the next subsections we elaborate on the main difficulties that arise and our strategy to overcome them.\\ 
Unlike the work of Fr\"{o}hlich and Gang \cite{Fro2}, we do not use any contour deformation technique to estimate the various integrals that appear, and therefore do not require the potential to be analytic. Instead our methods are purely dispersive. This is partly motivated by the fact that such methods proved successful to study systems of particle-gas interaction type in the seminal work of A. Soffer and M. Weinstein in \cite{SWnonlin} and \cite{SWlin}. Indeed the difficulty in our problem is realted to the presence of threshold eigenvalues: Cherenkov radiations happen when the particle (the eigenvalue) interacts with the continuous spectrum of the linear operator that dictates the evolution of the field. This scenario corresponds to the particle having velocity above the speed of sound. It means that if the particle has velocity \emph{exactly} equal to the speed of sound, we are in the situation of a \emph{threshold eigenvalue.} This is notoriously hard to analyse, and the methods of analytic deformations are known not to be well-suited for their study. This was one of the motivations of A. Soffer and M. Weinstein for their development of a purely dispersive method to tackle such embedded eigenvalue problems (see the introduction of their paper \cite{SWlin}).\\
This is why Fr\"{o}hlich and Gang require the particle to remain in the supersonic range. Although its velocity tends to the speed of sound, their assumption that the initial velocity is well-calibrated (not too close to the speed of sound and not too large) allows them to control very precisely the position of the particle, and gives in particular an explicit lower bound on its distance to the speed of sound (threshold between subsonic and supersonic regimes). \\
In this paper however we allow the particle to move from one regime to the other, and its velocity is therefore allowed to be equal to the speed of sound. As stated above, our method is purely dispersive. However a major drawback is that is requires to work in space dimension 5 or higher. We discuss this in more detail in the next subsection.   

\subsection{The space dimension 5 assumption} \label{disc}
We now explain the relevance of the dimension five assumption. Such a condition was also present in the work on quantum resonance theory of A. Soffer and M. Weinstein \cite{SWlin} where the case of a threshold eigenvalue is considered in their example 5, Section 6. Note that this was out of reach from previous contour deformation methods. The method however requires the space dimension to be five or higher.  \\
Let's explain this difficulty, which is not inherent to their method, but comes up as soon as threshold eigenvalues are present. In this context we must estimate remainder terms that are oscillatory integrals of the form 
\begin{align*}
\int_{\mathbb{R}^d} \frac{e^{it \phi_1(t,\xi)}}{\phi_2 (t,\xi)} f(\xi) d\xi 
\end{align*}
where $\phi_1$ and $\phi_2$ denote phase functions. For such an example in our work see the term $R_4$ in Lemma \ref{reform}. \\
The fact that we are considering embedded eigenvalues translates into the phase $\phi_2$ vanishing. The integrand is then understood in a principal value sense. For the expected scattering to happen, we need the above integral to have integrable decay (in our paper for example the acceleration of the particle is written in terms of such integrals in Lemma \ref{reform}, therefore integrable decay implies that the velocity has a limit as $t \to \infty$). Let's write heuristically 
\begin{align*}
\int_{\mathbb{R}^d} \frac{e^{it \phi_1(t,\xi)}}{\phi_2 (t,\xi)} f(\xi) d\xi =(-i) \int_{0}^{\infty} \int_{\mathbb{R}^d} e^{it \phi_1(t,\xi)} e^{i \tau \phi_2 (\xi)} f(\xi) d\xi d\tau
\end{align*}
Now notice that if the stationary points of the two phases $\phi_1$ and $\phi_2$ are far enough apart, then one would expect a decay like $t^{-d/2}$ where $d \geqslant 3$ denotes the space dimension for the whole integral. However if the two phases are equal, then the decay is  worse: in this case we should only expect a decay like $t^{1-\frac{d}{2}}.$ In particular the term is not integrable if $d \leqslant 4.$ (For a rigorous proof of a closely related result, see the work of F. Bernicot and P. Germain in \cite{B}, theorem 3.1 (iii), (iv)). 
\\
In our problem, the situation $\phi_1 = \phi_2$ in the expression of $R_4$ happens if $X(t) = t P(t).$ This is typically true in the case where the particle has ballistic motion ($X(t) = t P_{\infty}$, $P(t) = P_{\infty}$ where $P_{\infty}$ is constant). In fact the problem under study here is worse since there are remainder terms that decay even slower than $R_4$ (see Section \ref{acceleration} for more details), hence the necessity of the dimension five assumption.
\\
We end this section with another related situation, the so-called double Duhamel trick. It is used to make sense of $L^2$ norms of expressions like $(-\Delta-E)^{-1} f$ where $E$ lies in the continuous spectrum of $-\Delta.$ As we said, such expressions are common in wave-particle systems since they model the interaction of the particle with the field. At first glance it seems difficult to even make sense of such a norm. But, as T. Tao explains in the third section of his paper \cite{Tao}, we replace the principal values by well-chosen oscillatory integral:
\begin{align*}
\Vert (-\Delta-E)^{-1} f \Vert^2 _{L^2(\mathbb{R}^d)} & \leqslant \int_0 ^{\infty} \int_{-\infty}^0 \vert \lbrace e^{it\Delta} f , e^{it' \Delta} f \rangle \vert ~dt dt' \\
                                                      & \lesssim \int_0 ^{\infty} \int_{-\infty}^0 \frac{1}{\langle t-t' \rangle ^{d/2}} \Vert f \Vert_{L^1(\mathbb{R}^d) \cap L^2(\mathbb{R}^d)} ~dt dt'.
\end{align*}
That last integral converges for $d \geqslant 5.$ In our situation we encounter a similar problem in that expressions like squares of principal values appear naturally (at least heuristically, see the introductory part of Section \ref{reform}. Moreover they are always paired against smooth functions. By placing one of the principal values on the smooth function (we can work in Fourier space by Plancherel's theorem), this is morally equivalent to bounding the $L^2-$ norm of a principal value. This is another reason why the space dimension 5 seems relevant in our problem.

\subsection{Strategy of the proof}
In this subsection we describe the plan of the proof of Theorem \ref{main}. There are three main steps:

\subsubsection*{Step 1: Reformulation of the equations} First we find a suitable equation satisfied by the velocity of the particle. Since our goal is to prove that the motion is ballistic asymptotically, we expect the acceleration of the particle to decay. However the field is supposed to approach the shape of a soliton and thus not decay in time. As a result there no hope to deduce directly from the evolution equation of the particle in \eqref{system} that $\dot{P}$ decays. \\
The same issue was encountered by Fr\"{o}hlich and Gang in \cite{Fro2}, and they also derived and worked with a different equation for the velocity of the particle. However we do not use the same equation as them, since it includes terms that are not well defined (much less decay) if we wish to allow the particle to move between the supersonic and subsonic regimes (see Section \ref{setup} for details, in particular the introductory part). Our equation does however retain a key feature from \cite{Fro2}, namely the friction term, that encodes the emission of Cherenkov radiation. It is the leading order term of the equation, and its negative sign forces the velocity to decay if the particle is in the supersonic regime. 
\\ The way we derive this term however is closer in spirit to the work of A. Soffer and M. Weinstein on quantum resonances in \cite{SWlin}. Motivated by this paper, we expect the leading term to appear as the interaction of the particle with the continuous spectrum of the linear operator that gives the evolution of the field (seen in the frame of the particle and denoted $H(t)$ in the sequel). Hence we work with a new variable (namely $H(t) h,$ see the proof of Lemma \ref{reform} in Section \ref{setup} for details) that makes this coupling term appear clearly. As we said the equation also has the advantage of having less singular remainder terms. At the end we obtain an equation of the form
\begin{align} \label{velocsimp}
\dot{P} = g \big( \vert P \vert -1 \big) + \textrm{Remainder terms}
\end{align}
where the first term $g\big( \vert P \vert -1 \big)$ represents the friction of the system. The function $g$ is equal to 0 on $(-\infty;0]$ and strictly positive in $(0;+\infty)$ (it is discontinuous at 0). This is consistent with the intuition that the particle should decelerate in the supersonic range, and not be affected by the friction when it is subsonic.
\\ 
Another useful aspect of the equation is that it does not involve the evolution of the field. This allows us to treat the evolution of the particle independently from that of the field.

\subsubsection*{Step 2: Decay estimate for the acceleration of the particle} In Section \ref{acceleration} we analyse the equation obtained in the previous step and show that the acceleration of the particle decays at a rate $(1+t)^{-1-\alpha},$ for some $\alpha >0.$ This integrable decay implies the asymptotic ballistic motion of the particle announced in Theorem \ref{main}. The proof uses in a crucial way the friction in the system, though in a different way than in \cite{Fro2}. The idea used in that paper is to consider the full equation \eqref{velocsimp} as a perturbation of $\dot{P} = g \big( \vert P \vert -1 \big). $ \\ 
In the situation under study here however we essentially only use the negative sign of the friction. This non-perturbative approach  allows us to treat more general initial velocities than in \cite{Fro2}. The bulk of this section is devoted to proving that this term is indeed the leading order term in the equation of the velocity, and that the other terms present (refered to as remainder terms in \eqref{velocsimp}) decay fast enough in time. Since these lower-order terms are oscillatory in time, the decay is proved through stationary phase analysis. The difficulty here is that, as it is the case in particle-field interaction systems, the decay is slower than when coupling is absent, see our earlier discussion on the dimension five assumption. In fact, as we mentioned, even the milder terms require the dimension to be five or higher. We will see that there is in fact a term that is worse and cannot be treated through a simple application of the stationary phase lemma, even in dimension five (it would only provide a decay like $t^{-1/2}$). To overcome this issue we study more carefully the phase of the term in question, and localize precisely its stationary point. We prove that it is close to the origin, which allows us to leverage the cancellation of the Fourier transform of the interaction potential (the Golden rule in our setting) to improve the decay estimate. On the technical side we note that our proof allows for a very weak version of the Golden rule to be satisfied: the cancellation of the Fourier transform of the potential can be fixed arbitrarily small (see Remark \ref{remarquefermi}).
\subsubsection*{Step 3: Asymptotic behavior of the field} Finally we come to the scattering of the field in Section \ref{scattering}. Once we have established the convergence of the velocity of the particle, and found the decay rate of its acceleration, the asymptotic behavior of the field follows rather quickly. We can use our knowledge of the asymptotic behavior of the particle to find an appropriate guess for the long-time soliton behavior of the field. Indeed we expect the soliton part of the field to be located around the position of the particle. To find its profile we simply pass to the limit as $t \to \infty$ in the evolution equation of the field seen in the moving frame of the particle: the equation of the field is of the form
\begin{align*}
\partial_t h = H(t) h + W
\end{align*}
where 
\begin{displaymath}
H(t) := \left( \begin{array}{ll}
P(t) \cdot \nabla & - \Delta \\
\Delta - 1 & P(t) \cdot \nabla
\end{array} \right) 
\end{displaymath}
Therefore the profile $S$ of the soliton should solve
\begin{align*}
0 = H(\infty) S + W
\end{align*}
where 
\begin{displaymath}
H(\infty) := \left( \begin{array}{ll}
P_{\infty} \cdot \nabla & - \Delta \\
\Delta - 1 & P_{\infty} \cdot \nabla
\end{array} \right) 
\end{displaymath}
Then we estimate the difference between the field and this guess using the decay of the acceleration of the particle proved in the previous step. We conclude that this difference converges to 0 in $L^{\infty}$ which finishes the proof. 

\subsection{Notations and main result} \label{results}
We end this introduction with a more precise statement of our main theorem. First we introduce some notations:
\begin{itemize}
\item We will use the following convention for the Fourier transform:
\begin{align*}
\mathcal{F} f (\xi)= \widehat{f}(\xi) = \int_{\mathbb{R}^5} e^{-i x \cdot \xi} f(x) dx
\end{align*}
therefore the inverse Fourier transform is defined as 
\begin{align*}
\mathcal{F}^{-1} f(x) = \frac{1}{(2 \pi)^5} \int_{\mathbb{R}^5} e^{i x \cdot \xi} f(\xi) d\xi
\end{align*}
\item The operator $ L := \sqrt{(-\Delta)(1-\Delta)} $ will also appear repeatedly in the paper. We will also write $ \phi_1 (\xi) := \mathcal{F} L (\xi) = \vert \xi \vert \sqrt{1+\vert \xi \vert^2} $. \\
\item We also introduce the operator $U = \displaystyle \frac{\sqrt{-\Delta}}{\sqrt{1-\Delta}}.$ This operator appears in the decay estimates for the semi-group $e^{itL}$ in Lemma \ref{classicdisp} in the appendix.\\ 
\item We define the following inner product: $\langle f , g \rangle = \Re \int_{\mathbb{R}^3} \bar{f} g . $ 
\item Since the system \eqref{system} is Hamiltonian, a natural space to consider is the energy space, denoted $\mathcal{E}$. It is the set of functions such that the Hamiltonian \eqref{hamiltonian} is well-defined and finite. We denote $ \Vert \cdot \Vert_{\mathcal{E}} := \Vert \Re \cdot \Vert_{H^1(\mathbb{R}^5)} + \Vert \Im \cdot \Vert_{\dot{H}^1(\mathbb{R}^5)} $ the norm of that space. 
\item The potential of interaction $W$ is assumed to have the following special form: $ W = (-\Delta)^n V $ where $\widehat{V}(0) \neq 0$ and $n > 1/4.$  This is the so-called Fermi Golden rule. This is already present in Fr\"{o}hlich and Gang's paper \cite{Fro2}. Note that we kept the parameter $\rho_0$ in the Hamiltonian \ref{hamiltonian} so that we can impose the condition that $\vert \widehat{V}(0) \vert = 1. $
\\
In addition to being assumed of the form $W=(-\Delta)^n V$ with $\vert \widehat{V}(0) \vert = 1 $, the potential considered will be very regular and localized. We use the following norm to quantify this information: 
\begin{align*}
&\Vert W \Vert_{n}' := \big \Vert W \big \Vert_{\textbf{W}^{4,1}} + \big \Vert (1+\vert x \vert^4) V \big \Vert_{H^{4n+4}} + \Vert W \Vert_{H^4}
\end{align*}
where $H^s$ and $\textbf{W}^{4,1}$ denote the usual Sobolev spaces.  \\
We will omit the subscript $n$ to simplify notations. 
\item We denote $\mathcal{P}_n$ the set of functions $W$ of the form $W=(-\Delta)^n V$ where $\vert \widehat{V}(0) \vert = 1 $, and such that $\Vert W \Vert'<\infty.$ 
\item We will use the notation $W^X$ to denote the function $x \mapsto W(X-x).$
\item Finally we will need the following norm, useful to write decay estimates for the field:
\begin{align*}
\vert \vert \vert \beta_0 \vert \vert \vert := \Vert \beta_0 \Vert_{\mathcal{E}} + \Vert U^{3/2} \Re \beta_0 \Vert_{L^1} + \Vert U^{1/2} \Im \beta_0 \Vert_{L^1}  + \Vert U^{5/2} \Im \beta_0 \Vert_{L^{1}}  
\end{align*}
\end{itemize}
\bigskip
Let's now state precisely the main result of our paper. It says that asymptotically, the particle has a ballistic motion, and that the field scatters. More precisely we prove that it is well approximated (in $L^{\infty}$) by a traveling wave solution of the system:
\begin{theorem} \label{main}
Let $n > 1/4$ be a real number. \\
Let $W \in \mathcal{P}_n.$ Let $\beta_0 \in \mathcal{E}$ be such that $\vert \vert \vert \beta_0 \vert \vert \vert < \infty.$ Let $P_0 \in \mathbb{R}^5$ denote the initial velocity.  \\
Then there exists $1>\varepsilon_0 >0$ such that if $\rho_0<\varepsilon_0,$ there exists $\theta_0, \theta_0 ' >0$ such that if $\Vert W \Vert' < \theta_0$ and $\vert \vert \vert \beta_0 \vert \vert \vert < \theta_0'$, then scattering holds, in the sense that there exists $ (P_{\infty}, \beta_{\infty}) \in \mathbb{R}^5 \times  L^{\infty}(\mathbb{R}^5) $ such that \\
\begin{eqnarray*}
P(t) \longrightarrow_{t \rightarrow + \infty} P_{\infty} \\
\Vert \beta(t) - \beta_{\infty}(\cdot - X(t)) \Vert_{{\infty}} \longrightarrow_{t \rightarrow + \infty} 0
\end{eqnarray*}
and X(t) denotes the position of the particle. \\
Moreover $\beta_{\infty}$ solves the following linear elliptic equation:
\begin{displaymath} \left( \begin{array}{cc}
P_{\infty} \cdot \nabla & - \Delta \\
\Delta -1 & P_{\infty} \cdot \nabla
\end{array} \right) \beta_{\infty} = \rho_0 
\left( \begin{array}{ll}
0 \\
W
\end{array} \right)
\end{displaymath}
\end{theorem}
\begin{remark}
The smallness condition means that we require \begin{align*}
\Vert W \Vert', \vert \vert \vert \beta_0 \vert \vert \vert \ll \rho_0 \ll 1
\end{align*}
\end{remark}
\bigskip
\textbf{Acknowledgments:} The author thanks his PhD advisor Prof. Pierre Germain for suggesting this problem to him as well as for many very interesting discussions about it. He also thanks him for his helpful comments on an earlier version of this paper.

\section{Set-up} \label{setup}
First we write the equations in the frame of the moving particle: \\
Consider the new unknown 
\begin{displaymath}
h(x-X)= \left(
\begin{array}{ll}
\Re \beta \\
\Im \beta
\end{array} \right)
\end{displaymath}
The system \eqref{system} becomes:
\begin{equation}\label{systembis} 
\begin{cases}
\dot{X}(t) =  P(t)  \\
\dot{P}(t) =  \sqrt{{\rho}_0} \Re \Big \langle  
\left( \begin{array}{ll}
\nabla W \\
0 
\end{array} \right)
 , h \Big \rangle \\
\partial_t h = H(t) h - \sqrt{\rho_0} 
\left( \begin{array}{ll}
0 \\
W
\end{array} \right)
\end{cases}
\end{equation}
where 
\begin{displaymath}
H(t) := \left( \begin{array}{ll}
P(t) \cdot \nabla & - \Delta \\
\Delta - 1 & P(t) \cdot \nabla
\end{array} \right) 
=: P(t) \cdot \nabla ~I_{2} + H_0 
\end{displaymath}
where $I_2$ denotes the $2 \times 2$ identity matrix.

\subsection{Preliminaries}
We start by recalling a global well-posedness result for \eqref{system}, which is a straightforward adaptation to five space dimension of theorem 2.2 from \cite{E} (see appendix A of that same paper for a proof):
\begin{theorem}
Let $W \in \mathcal{P}.$ Let $P_0 \in \mathbb{R}^5$ and $\beta_0 \in \mathcal{E}.$ \\
Then the system \eqref{system} has a unique global solution 
\begin{align*}
X(t),P(t) \in \mathcal{C}([0;\infty); \mathbb{R}^5) \\
\beta \in \mathcal{C}([0; \infty); \mathcal{E})
\end{align*}
with initial condition $(P_0, \beta_0).$
\end{theorem}
As mentioned in the introduction, this system is Hamiltonian for 
\begin{displaymath}
\mathcal{H}(X, P, \beta, \nabla \beta) = \frac{P^2}{2} + \int_{\mathbb{R}^3} \vert \nabla \beta \vert^2 dx + \lambda \int_{\mathbb{R}^3} \vert \Re \beta \vert^2 dx + 2 \sqrt{\rho _0} \int_{\mathbb{R}^3} W^{X} \Re \beta dx
\end{displaymath}
Using this conservation law we easily obtain a uniform bound on the velocity of the particle, as well as $\Vert \Re \beta \Vert_{L^2}$ and $\Vert \nabla \Im \beta \Vert_{L^2}$ :
\begin{lemma}\label{easybound}
We have the bounds
\begin{align*}
\sup_{t \geqslant 0} \vert P(t) \vert & \leqslant \sqrt{2 \mathcal{H} + \sqrt{\rho_0} \Vert W \Vert_2 ^2} \\
\sup_{t \geqslant 0} \Vert \Re \beta \Vert_2 & \leqslant \sqrt{ \frac{1}{1-\sqrt{\rho_0}} \bigg(  \mathcal{H} + \sqrt{\rho_0} \Vert W \Vert_2 ^2 \bigg) } \\
\sup_{t \geqslant 0} \Vert \nabla \Im \beta \Vert_2 & \leqslant \sqrt{\mathcal{H} + \sqrt{\rho_0} \Vert W \Vert_2 ^2} 
\end{align*}
\end{lemma}
\begin{proof}
By conservation of the Hamiltonian, we have
\begin{align*}
\mathcal{H} &= \frac{\vert P \vert^2}{2} + \int_{\mathbb{R}^3} \vert \nabla \beta \vert^2 dx + \int_{\mathbb{R}^3} \vert \Re \beta \vert^2 dx + 2 \sqrt{\rho_0} \int_{\mathbb{R}^3} W^X \Re \beta dx \\
& \geqslant \frac{\vert P \vert^2}{2} + (1-\sqrt{\rho_0}) \int_{\mathbb{R}^3} \vert \Re \beta \vert^2 dx -  \sqrt{\rho_0} \Vert W \Vert_2 ^2 + \int_{\mathbb{R}^5} \vert \nabla \Im \beta \vert^2 dx
\end{align*}
Hence 
\begin{align*}
\vert P \vert^2 & \leqslant 2 \mathcal{H} + \sqrt{\rho_0} \Vert W \Vert_2 ^2 \\
\Vert \Re \beta \Vert_2 ^2 & \leqslant \frac{1}{1-\sqrt{\rho_0}} \bigg(  \mathcal{H} + \sqrt{\rho_0} \Vert W \Vert_2 ^2 \bigg) \\
\Vert \nabla \Im \beta \Vert_2 ^2 & \leqslant  \mathcal{H} + \sqrt{\rho_0} \Vert W \Vert_2 ^2 
\end{align*}
We deduce from these three inequalities the desired bounds. 
\end{proof} 
Now we notice that we can actually improve the regularity of the velocity function, and prove that it is $\mathcal{C}^1:$
\begin{corollary}\label{C1}
Consider $(P(t),\beta(t))$ the unique solution given by the previous theorem. \\
Then we have that $P \in \mathcal{C}^1([0;\infty); \mathbb{R}^5)$
\end{corollary}
\begin{proof}
We have, for $t,t' \geqslant 0$ that
\begin{eqnarray*}
\big \vert \dot{P}(t') - \dot{P}(t) \big \vert &=& \bigg \vert \int_{t} ^{t'} \Big \langle 
\left( \begin{array}{cc}
\nabla W \\
0
\end{array} \right) , \partial_t h \Big \rangle ds \bigg \vert \\
&\leqslant & \int_{t} ^{t'} \bigg \vert \Big \langle 
\left( \begin{array}{cc}
\nabla W \\
0
\end{array} \right) , H(t) h \Big \rangle \bigg \vert ds \\ 
& \leqslant & \vert t-t' \vert \sup_{t} \bigg \vert \langle \partial_i W , P(t) \cdot \nabla \Re h - \Delta \Im h \rangle \bigg \vert  
\end{eqnarray*}
Now we use integration by parts to place derivatives on $W,$ as well as the Cauchy-Schwarz inequality to obtain:
\begin{align*}
\vert \dot{P}(t) - \dot{P}(t') \vert & \leqslant \vert t-t' \vert \bigg( \sup_t \bigg \vert \langle P(t) \cdot \nabla \partial_i W , \Re h \rangle \bigg \vert + \sup_t \bigg \vert \Delta \partial_i W , \Im h \rangle \bigg \vert \bigg)
\\ 
                                     & \leqslant 2 \vert t-t' \vert \sup_{t} \big( \vert P(t) \vert \big) \Vert W \Vert' \big( \Vert \Re \beta \Vert_{L^2} + \Vert \nabla \Im \beta \Vert_{L^2} \big)
\end{align*} 
and we conclude using Lemma \ref{easybound}.
\end{proof}

\subsection{Recasting the equation in a favorable form} 
In this subsection we derive a new evolution equation for the particle.\\
Since $h$ satisfies an equation with a forcing term (see \eqref{systembis}), it would be natural to write it as
\begin{align*}
h(t) = H^{-1} (t) \bigg( \begin{array}{ll}
0 \\
W
\end{array} \bigg)
+ \delta(t)
\end{align*}
where $H^{-1}(t)$ is understood in Fourier space as 
\begin{displaymath} P.V. \left( \begin{array}{cc}
i P(t) \cdot \xi      &     \vert \xi \vert^2 \\
-1-\vert \xi \vert^2 & i P(t) \cdot \xi
\end{array} \right)^{-1}
\end{displaymath}
where $P.V.$ stands for principal value. \\
The first part would stand for the soliton part, and the $\delta$ for the radiative part. We would then plug this ansatz back into the equation. This is the strategy followed by Fr\"{o}hlich and Gang in \cite{Fro2}.
\\
However we would then have to differentiate the soliton term in time which makes a singularity appear (formally the derivative of $H^{-1} (t)$ is $\dot{H}(t) H^{-2} (t)$). Fr\"{o}hlich and Gang are able to tackle this issue by having a very precise control over the position of the particle, hence the need to be in a very specific regime. \\
\\
In this paper we take an approach that does not introduce a singularity like $H^{-2} (t),$ but retains the main physical feature of the system, namely the friction. The next lemma, which constitutes the main result of this section, decomposes the acceleration of the particle into three main parts: the friction term (which is the only term explicitely given in the formula below), the remainder terms (denoted $R_1, R_2, R_3, R_4$) and the regularization terms ($E_{1,\epsilon}, E_{2,\epsilon}$). We will show later on that the friction term is the leading term in the equation. More precisely we will prove that the remainder terms have integrable decay in time. The regularization terms are only there for technical reasons, namely to ensure that principal values are well-defined and that our computations are rigorous.
\\
Another favorable aspect of the equation we give below is that it does not involve that evolution of the field. In a sense we decoupled the evolution of the particle from that of the field, which allows us to analyse the two independently.
\\
The exact equation is given in the following lemma:
\begin{lemma}\label{reform}
Let $\epsilon >0.$ \\
We have
\begin{align*}
\dot{P}(t) &= \rho_0 \Re \Big \langle \Bigg( \begin{array}{c}
\nabla_x W \\
0
\end{array} \Bigg) , H^{-1}_{\epsilon} (t) \Bigg( \begin{array}{c}
0 \\
W
\end{array} \Bigg) \Big \rangle + R_1 + R_2 + R_3  + R_4 + E_{1,\epsilon} + E_{2,\epsilon}
\end{align*}
where 
\begin{align*}
H_{\epsilon} (t) = \Bigg( \begin{array}{cc}
P(t) \cdot \nabla - \epsilon & - \Delta \\
\Delta -1 & P(t) \cdot \nabla_x - \epsilon
\end{array} \Bigg) = H_0 + (P(t) \cdot \nabla_x - \epsilon) I_2
\end{align*}
and 
\begin{eqnarray*}
R_1 &=& \frac{\sqrt{\rho_0}}{2(2\pi)^5} \Re \Big \langle i \xi_j \widehat{W} , \widehat{ \Re \beta_0} D_1 ^1 + \frac{\widehat{\Im \beta_0} \vert \xi \vert}{i \sqrt{1+\vert \xi \vert ^2}} D_1 ^1 +  \widehat{\Re \beta_0} D_2 ^1- \frac{\widehat{\Im \beta_0} D_2 ^I \vert \xi \vert}{i \sqrt{1+\vert \xi \vert^2}} \Big \rangle
\end{eqnarray*} 
With
\begin{eqnarray*}
D_1 ^1 &=& \frac{(i\phi_1 (\xi) + i P(0) \cdot \xi )e^{it \phi_1 (\xi) + i (X(t)-X(0)) \cdot \xi}}{i\phi_1(\xi)+i P(t) \cdot \xi-\epsilon} \\
D_2 ^1 &=& \frac{(-\phi_1 (\xi) + i P(0) \cdot \xi )e^{-it \phi_1 (\xi) + i (X(t)-X(0)) \cdot \xi}}{-\phi_1(\xi)+i P(t) \cdot \xi - \epsilon}
\end{eqnarray*}
where $\phi_1(\xi) = \vert \xi \vert \sqrt{1+\vert \xi \vert^2}$\\
and
\begin{eqnarray*}
E_{1,\epsilon} &=& \frac{\sqrt{\rho_0}}{2(2\pi)^5} \Re \Big \langle i \xi_j \widehat{W} , \widehat{\Re \beta_0} D_1 ^{1,\epsilon} + \frac{\widehat{\Im \beta_0} \vert \xi \vert}{i \sqrt{1+\vert \xi \vert ^2}} D_1 ^{1,\epsilon} +  \mathcal{F} \Re \beta_0 D_2 ^{1,\epsilon}- \frac{\mathcal{F} \Im \beta_0 D_2 ^{1,\epsilon} \vert \xi \vert}{i \sqrt{1+\vert \xi \vert^2}} \Big \rangle
\end{eqnarray*} 
With
\begin{eqnarray*}
D_1 ^{1,\epsilon} &=& -\epsilon\frac{ e^{it \phi_1 (\xi) + i (X(t)-X(0)) \cdot \xi}}{i\phi_1(\xi)+i P(t) \cdot \xi-\epsilon} \\
D_2 ^{1,\epsilon} &=& -\epsilon\frac{ e^{-it \phi_1 (\xi) + i (X(t)-X(0)) \cdot \xi}}{-\phi_1(\xi)+i P(t) \cdot \xi - \epsilon}
\end{eqnarray*}
and
\begin{eqnarray*}
R_2 &=& \frac{\sqrt{\rho_0}}{2(2 \pi)^5} \Re \Big \langle i \xi_j \widehat{W} , \widehat{\Re h_0} D_1 ^{2} + \frac{\widehat{\Im h_0} \vert \xi \vert}{i \sqrt{1+\vert \xi \vert ^2}} D_1 ^{2} +  \widehat{\Re h_0} D_2 ^{2}- \frac{\widehat{\Im h_0} D_2 ^{2} \vert \xi \vert}{i \sqrt{ 1+ \vert \xi \vert^2}} \Big \rangle
\end{eqnarray*} 
With
\begin{eqnarray*}
D_1 ^{2} &=& \frac{e^{it \phi_1 (\xi) + i (X(t)-X(0)) \cdot \xi}}{i\phi_1(\xi)+i P(t) \cdot \xi- \epsilon} i \int_0 ^t \dot{P} (s) \cdot \xi ds  \\
D_2 ^{2} &=&  \frac{e^{-it \phi_1 (\xi) + i (X(t)-X(0)) \cdot \xi}}{-i\phi_1(\xi)+i P(t) \cdot \xi-\epsilon} i \int_0 ^t \dot{P} (s) \cdot \xi ds
\end{eqnarray*}
and 
\begin{eqnarray*}
R_3 &=& - \frac{\rho_0}{2(2\pi)^5} \Re \Big \langle i \xi_j \widehat{W} ,  \frac{\widehat{W} \vert \xi \vert}{i \sqrt{1+\vert \xi \vert ^2}} D_1 ^{3}  - \frac{\widehat{W} D_2 ^{3} \vert \xi \vert}{i \sqrt{1+\vert \xi \vert^2}} \Big \rangle
\end{eqnarray*} 
With
\begin{eqnarray*}
D_1 ^{3} &=& \frac{e^{it \phi_1 (\xi) + i (X(t)-X(0)) \cdot \xi}}{i\phi_1(\xi)+i P(t) \cdot \xi-\epsilon} i \int_0 ^t \int_0 ^s \dot{P} (s) \cdot \xi e^{-i \tau \phi_1 (\xi) - i (X(\tau)-X(0)) \cdot \xi} ds d\tau \\
D_2 ^{3} &=& \frac{e^{-it \phi_1 (\xi) + i (X(t)-X(0)) \cdot \xi}}{-i\phi_1(\xi)+i P(t) \cdot \xi - \epsilon} i \int_0 ^t \int_0 ^s \dot{P} (s) \cdot \xi e^{i \tau \phi_1 (\xi) - i (X(\tau)-X(0)) \cdot \xi} ds d\tau
\end{eqnarray*}
and 
\begin{eqnarray*}
R_4 &=&  \frac{\rho_0}{2(2\pi)^5} \Re \Big \langle i \xi_j \widehat{W} , \frac{\widehat{W} \vert \xi \vert}{i \sqrt{1+\vert \xi \vert^2}} D_1 ^4 - \frac{\widehat{W} D_2 ^4 \vert \xi \vert}{i \sqrt{1+ \vert \xi \vert^2}} \Big \rangle
\end{eqnarray*} 
With
\begin{eqnarray*}
D_1 ^{4} &=& \frac{e^{it \phi_1 (\xi) + i (X(t)-X(0)) \cdot \xi}}{i\phi_1(\xi)+i P(t) \cdot \xi-\epsilon} \\
D_2 ^{4} &=& \frac{e^{-it \phi_1 (\xi) + i (X(t)-X(0)) \cdot \xi}}{-i\phi_1(\xi)+i P(t) \cdot \xi-\epsilon}
\end{eqnarray*}
and 
\begin{eqnarray*}
E_{2,\epsilon} &=&  \frac{\rho_0}{2(2\pi)^5} \Re \Big \langle i \xi_j \widehat{W} , \frac{\widehat{W} \vert \xi \vert}{i \sqrt{1+\vert \xi \vert^2}} D_1 ^{2,\epsilon} - \frac{\widehat{W} D_2 ^{2,\epsilon} \vert \xi \vert}{i \sqrt{1+\vert \xi \vert^2}} \Big \rangle
\end{eqnarray*} 
With
\begin{eqnarray*}
D_1 ^{2,\epsilon} &=& - \epsilon \frac{\int_0 ^t e^{i(X(t)-X(s)) \cdot \xi + (t-s) \phi_1 (\xi)}ds}{i\phi_1(\xi)+i P(t) \cdot \xi - \epsilon} \\
D_2 ^{2,\epsilon} &=&  - \epsilon\frac{\int_0 ^t e^{i(X(t)-X(s)) \cdot \xi + (t-s) \phi_1 (\xi)}ds}{-i\phi_1(\xi)+i P(t) \cdot \xi- \epsilon}
\end{eqnarray*}
\end{lemma}
\begin{remark}
There are only principal value type singularities present in this formulation.
\end{remark}
\begin{proof}
We start by writing down the expression of $h$ given by the Duhamel formula: 
\begin{eqnarray}\label{Duhamel}
h(t) = e^{R(t)} h_0 - \sqrt{\rho_0} \int _0 ^t e^{R(t)-R(s)} \left( \begin{array}{ll}
0 \\
W
\end{array} \right) ds
\end{eqnarray}
with 
\begin{displaymath}
R(t) = \int _0 ^t H(s) ds
\end{displaymath}
Now notice that $H_{\epsilon} (t)$ can be diagonalized using Lemma \ref{diag} from the appendix: 
\begin{align*}
H_{\epsilon} (t) = A \Bigg( \begin{array}{cc}
iL + P(t) \cdot \nabla - \epsilon & 0 \\
0 & -iL + P(t) \cdot \nabla - \epsilon
\end{array}
 \Bigg) A^{-1}
\end{align*}
Looking at the Fourier transform of this matrix, we see that it is invertible and that its inverse is defined in Fourier space by
\begin{align*}
\mathcal{F} H^{-1} _{\epsilon} (t) = \mathcal{F} A \Bigg( \begin{array}{cc}
\frac{1}{i \vert \xi \vert \sqrt{1+\vert \xi \vert^2} + i \xi \cdot P(t) - \epsilon} & 0 \\
0 & \frac{1}{-i \vert \xi \vert \sqrt{1+\vert \xi \vert^2} + i \xi \cdot P(t) - \epsilon}
\end{array}
\Bigg) \mathcal{F}A^{-1}
\end{align*}
Now we turn to the equation satisfied by the velocity of the particle. We write that
\begin{align}
\notag \dot{P}(t) & =  \sqrt{\rho _0} \Re \Big \langle \left( \begin{array}{ll}
\nabla W \\
0
\end{array} \right), h(t) \Big \rangle \\
\label{pdot}&= \sqrt{\rho _0} \Re \Big \langle \left( \begin{array}{ll}
\nabla W \\
0
\end{array} \right),H_{\epsilon}^{-1}(t) H_{\epsilon}(t) h(t) \Big \rangle
\end{align}
Now we derive an equation for $\eta(t) = H_{\epsilon}(t) h(t): $ 
\begin{eqnarray*}
\partial_t \eta &=& (\dot{P} \cdot \nabla)  h(t) + H_{\epsilon}(t) \partial _t h \\
&=&   (\dot{P} \cdot \nabla)  h(t) + H(t) \eta -\sqrt{\rho _0} H_{\epsilon}(t) \left( \begin{array}{ll}
0 \\
W
\end{array} \right) 
\end{eqnarray*}
Using Duhamel's formula and \eqref{Duhamel} we get: 
\begin{eqnarray*}
\eta(t) &=& e^{R(t)} \eta _0 + \int_0 ^t e^{R(t)- R(s)} (\dot{P}(s) \cdot \nabla) h(s) ds - \sqrt{\rho _0} \int _0 ^t e^{R(t)-R(s)} H_{\epsilon}(s) \left( \begin{array}{ll}
0 \\
W
\end{array} \right) ds \\
&=& e^{R(t)} \eta _0 + \int _0 ^t e^{R(t)-R(s)} (\dot{P}(s) \cdot \nabla ) \Big[ e^{R(s)} h_0 - \sqrt{\rho _0} \int _0 ^s e^{R(s)-R(\tau)} \left( \begin{array}{ll}
0 \\
W
\end{array} \right) d \tau \Big] ds  \\
&-& \sqrt{\rho _0} \int _0 ^t e^{R(t) -R(s)} H_{\epsilon}(s) \left( \begin{array}{ll}
0 \\
W
\end{array} \right) ds 
\end{eqnarray*}
Now we plug this expression back into \eqref{pdot}:
\begin{eqnarray*}
\dot{P}(t) &=& \sqrt{\rho _0} \Re \Big \langle \left( \begin{array}{ll}
\nabla W \\
0
\end{array} \right) , H_{\epsilon} ^{-1} (t) e^{R(t)} \eta _0 \Big \rangle \\
&+&  \sqrt{\rho _0} \Re \Big \langle \left( \begin{array}{ll}
\nabla W \\
0
\end{array} \right) , H_{\epsilon}^{-1} (t) \int _0 ^t e^{R(t) -R(s)} (\dot{P}(s) \cdot \nabla)e^{R(s)} h_0 ds \Big \rangle 
\\
&-& \rho _0 \Re \Big \langle \left( \begin{array}{ll}
\nabla W \\
0
\end{array} \right) , H_{\epsilon}^{-1} (t) \int _0 ^t e^{R(t) -R(s)} (\dot{P}(s) \cdot \nabla) \int _0 ^s e^{R(s) -R(\tau)} \left( \begin{array}{ll}
0 \\
W
\end{array} \right) d \tau ds \Big \rangle 
\\
&-& \rho_0 \Re \Big \langle \left( \begin{array}{ll}
\nabla W \\
0
\end{array} \right)  , \int _0 ^t e^{R(t)-R(s)} H_{\epsilon}(s) ds~  H_{\epsilon}^{-1}(t) \left( \begin{array}{ll}
0 \\
W
\end{array} \right) \Big \rangle \\
&:=& I + R_2 + R_3 + IV
\end{eqnarray*}
Now we can diagonalize these terms using Lemma \ref{diag} from the appendix:
\\
\textit{Terms I, $R_2$, $R_3$:} These terms have the following form
\begin{eqnarray}\label{usefuldiag}
\Re \Big \langle  \left( \begin{array}{ll}
\nabla_x W \\
0 
\end{array} \right)
, A \left( \begin{array}{ll}
D_1 & 0 \\
0 & D_2 
\end{array} \right)
 A^{-1} \left( \begin{array}{ll}
a \\b
\end{array} \right)
\Big \rangle  \\
\notag= \frac{1}{2} \Re \Big \langle \nabla_x W , a D_1 + \frac{b \sqrt{-\Delta}}{i \sqrt{1-\Delta}} D_1 + a D_2- \frac{b D_2 \sqrt{-\Delta}}{i \sqrt{1-\Delta}} \Big \rangle
\end{eqnarray}
The formulas in the lemma follow directly from this. Note that $I = R_1+E_{1,\epsilon}$ where $E_{1,\epsilon}$ is the part that has size $\epsilon.$ \\
\textit{Term IV:} First we notice that 
\begin{align*}
& \mathcal{F} \Bigg( \int_0 ^t e^{R(t)-R(s)} H_{\epsilon}(s) ds \Bigg) \\
& = \mathcal{F}(A) \Bigg( \begin{array}{cc}
Z  &    0    \\
0    &   Z'
\end{array} \Bigg) \mathcal{F}(A^{-1})
\end{align*}
with
\begin{align*}
Z=\int_0 ^t  e^{i(X(t)-X(s)) \cdot \xi + (t-s) \phi_1 (\xi)} (i P(s) \cdot \xi + \phi_1 (\xi)-\epsilon)  ds
\end{align*}
and 
\begin{align*}
Z'=\int_0 ^t  e^{i(X(t)-X(s)) \cdot \xi - (t-s) \phi_1 (\xi)} (i P(s) \cdot \xi - \phi_1 (\xi)-\epsilon)  ds
\end{align*}

Now notice that
\begin{align*}
\int_0 ^t  e^{i(X(t)-X(s)) \cdot \xi + (t-s) \phi_1 (\xi)} (i P(s) \cdot \xi + \phi_1 (\xi))  ds &= \int_0 ^t -\frac{d}{ds} \big( e^{i(X(t)-X(s)) \cdot \xi + (t-s) \phi_1 (\xi)} \big) ds \\
& = -1 + e^{i X(t) \cdot \xi + t \phi_1 (\xi)}
\end{align*}
This gives us $IV = F + R_4 + E_{2,\epsilon}$ where $R_4$ stands for a remainder term, $E_{2,\epsilon}$ is a regularizing term of size $\epsilon$ and $F$ stands for the friction term:
\begin{align} \label{friction}
F = \rho_0 \Re \Big \langle \Bigg( \begin{array}{ll}
\nabla_x W \\
0
\end{array} \Bigg) , H_{\epsilon}^{-1} (t) \Bigg( \begin{array}{ll}
0 \\
W
\end{array} \Bigg) \Big \rangle
\end{align}
As for $I,R_2,R_3 $ we have for $R_4$, using \eqref{usefuldiag}:
\begin{eqnarray*}
R_4 &=&  \frac{\rho_0}{2} \Re \Big \langle i \xi_j \widehat{W} , \frac{\widehat{W} \sqrt{-\Delta}}{i \sqrt{1-\Delta}} D_1 ^{4} - \frac{\widehat{W} D_2 ^{4} \sqrt{-\Delta}}{i \sqrt{1-\Delta}} \Big \rangle
\end{eqnarray*} 
\end{proof}

\begin{remark}
Notice that unlike \eqref{friction}, the terms $R_1, R_2, R_3$ and $R_4$ are oscillatory in time. Heuristically, we therefore expect them to decay in time. They can be considered as remainder terms compared to the friction. We will make this precise in subsequent sections. 
\end{remark}
Now we have the following formula for the leading term: (this is a straightforward adaptation to five space dimension of Lemma 2.1 from \cite{Fro2}). It is at this point that we make use of the normalization $\vert \widehat{V}(0) \vert = 1.$ 
\begin{lemma}\label{representation}
We have
\begin{align*}
\lim_{\epsilon\to 0^+}\rho_0 \Re \Big \langle \Bigg( \begin{array}{c}
\nabla_x W \\
0
\end{array} \Bigg) , H^{-1}_{\epsilon} (t) \Bigg( \begin{array}{c}
0 \\
W
\end{array} \Bigg) \Big \rangle=-\rho_0 \Lambda(\vert P(t) \vert) (\vert P(t) \vert - 1)^{3+2n} \frac{P(t)}{\vert P(t) \vert} 1_{\lbrace \vert P \vert >1 \rbrace }
\end{align*}
where $ \Lambda $ is a discontinuous function at 1, smooth everywhere else and which is 0 for $ x \leqslant 1. $ It is bounded from below on $(1;+\infty)$ by a positive constant $ C_0 .$
\end{lemma}
\begin{remark}
The fact that $\Lambda$ is 0 in the subsonic regime indicates that the friction is only present in the supersonic regime.
\end{remark}
We conclude this section with the final reformulation of the equation satisfied by the velocity of the particle:
\begin{corollary}\label{finalform}
The velocity of the particle $P(t)$ satisfies the following ordinary differential equation:
\begin{align*}
\dot{P}(t) &= -\rho_0 \Lambda(\vert P(t) \vert) (\vert P(t) \vert - 1)^{3+2n} \frac{P(t)}{\vert P(t) \vert} 1_{\lbrace \vert P \vert >1 \rbrace } + R_1 + R_2 + R_3  + R_4 \\
&+ E_{1,\epsilon} +  E_{2,\epsilon}+R_{\epsilon} 
\end{align*}
where the explicit expression of $R_1, R_2, R_3$ and $R_4$ are given in Lemma \ref{reform} and 
\begin{align*}
R_{\epsilon} := \rho_0 \Lambda(\vert P(t) \vert) (\vert P(t) \vert - 1)^{3+2n} \frac{P(t)}{\vert P(t) \vert} 1_{\lbrace \vert P \vert >1 \rbrace } + \rho_0 \Re \Big \langle \Bigg( \begin{array}{c}
\nabla_x W \\
0
\end{array} \Bigg) , H^{-1}_{\epsilon} (t) \Bigg( \begin{array}{c}
0 \\
W
\end{array} \Bigg) \Big \rangle
\end{align*}
\end{corollary}
\begin{remark}
We isolated the terms $E_{1,\epsilon}, E_{2,\epsilon}$ and $R_{\epsilon}$ since they go to 0 as $\epsilon \to 0^+.$ They are irrelevant from the point of view of asymptotics. They are mere consequences of the need to regularize the singularity for the computations to be rigorous.
\end{remark}

\section{Ballistic motion of the particule} \label{acceleration}
In this section we describe the asymptotic behavior of the particle. More precisely we prove the following proposition, which is the first assertion in Theorem \ref{main}:
\begin{proposition} \label{mainprop}
We have the following estimate for all $t \geqslant 0:$
\begin{displaymath}
\vert \dot{P}(t) \vert \lesssim (1+t)^{-1-\alpha}
\end{displaymath}
where $ \alpha = \frac{1}{2n+2}$ and $n$ is, as we defined earlier, such that $W = (-\Delta)^n V.$ 
\end{proposition}
We have the immediate corollary:
\begin{corollary}
There exists $ P_{\infty} \in \mathbb{R}^5 $ such that 
\begin{displaymath}
P(t) \longrightarrow_{t \rightarrow + \infty} P_{\infty}
\end{displaymath}
\end{corollary}
\begin{remark}
The fact that the decay gets faster as the Golden Rule condition gets weaker (if $n$ decreases then $\alpha$ increases) is expected. Indeed the coupling causes anomalously slow decay, which was the main reason we needed the dimension five assumption. It is more generally true in field-particle interactions, see for example the work of A. Soffer and M. Weinstein in \cite{SWnonlin} where the solutions have slower decay than in the free case.  \\
For the technical justification of the choice of $\alpha$ see Remark \ref{alpha} below.
\end{remark}

\subsection{Short time estimate}
We start by giving a crude uniform in time upper bound of the acceleration of the particle: 
\begin{lemma}
We have the bound
\begin{align*}
\sup_{t \geqslant 0} \vert \dot{P}(t) \vert & \leqslant \sqrt{\frac{\rho_0}{1- \rho_0}  \bigg(\mathcal{H} + \sqrt{\rho_0}  \bigg)}
\end{align*}
\end{lemma}
\begin{proof}
We use the Cauchy-Schwarz inequality in the equation of the velocity of the particle in the original system \eqref{system} to write that
\begin{align*}
\vert \dot{P} \vert & \leqslant \sqrt{\rho_0} \Vert \nabla W \Vert_2 \Vert \Re \beta \Vert_{L^2} \\
                    & \leqslant \sqrt{\frac{\rho_0}{1- \rho_0}  \bigg(\mathcal{H} + \sqrt{\rho_0} \Vert W \Vert_2 ^2 \bigg)} \Vert \nabla W \Vert_2 \\
                    & \leqslant  \sqrt{\frac{\rho_0}{1- \rho_0}  \bigg(\mathcal{H} + \sqrt{\rho_0}  \bigg)}
\end{align*}
where for the second to last line we used the bound on $\Re \beta$ from Lemma \ref{easybound} and for the last line we used that $\varepsilon_0<1$ therefore $\Vert W \Vert' \leqslant 1.$ 
\end{proof}

\subsection{The bootstrap}
The proof of proposition \ref{mainprop} relies on a bootstrap argument. \\
Let 
\begin{align*}
M := 2 \max \bigg \lbrace  \sqrt{2 \mathcal{H} + \sqrt{\rho_0}} ;   \sqrt{\frac{\rho_0}{1- \rho_0}  \bigg(\mathcal{H} + \sqrt{\rho_0} \bigg)} ; \frac{51}{C_0 \rho_0} \bigg \rbrace
\end{align*}
where $C_0$ is the constant introduced in Lemma \ref{representation}.
\\ 
\\
Let $\alpha = \frac{1}{2n+2}. $\\
With these notations we make the following bootstrap assumptions: 
\begin{eqnarray}
\label{bootstrap1}\max_{0 \leqslant t \leqslant T} (1+t)^{1+ \alpha} \vert \dot{P}(t) \vert & \leqslant & M^{3+2n} \\
\label{bootstrap2}\max_{0 \leqslant t \leqslant T} (1+t)^{\alpha} (\vert P(t) \vert -1) 1_{\lbrace \vert P \vert > 1 \rbrace} & \leqslant & M
\end{eqnarray}
From the previous lemma, Lemma \ref{easybound} and by continuity in time of the norms of $P$ and $\dot{P},$ the bootstrap assumptions are true for small times. 
\\
Now let's start the proof of proposition 4.1.
\begin{proof}
Assume now that the assumptions are true on $[0;T]$. (we know that $T>0$ from our choice of $M$) \\
First we prove that we can continue the second assumption \eqref{bootstrap2} beyond $T.$  \\
\underline{Case 1:} Assume that $ (1+T)^{\alpha} (\vert P(T) \vert -1) 1_{\lbrace \vert P \vert > 1 \rbrace} < M. $ Then by continuity we can extend the interval on which the property is true, and the bootstrap is proved. \\
\underline{Case 2:} Assume $ (1+T)^{\alpha} (\vert P(T) \vert -1) 1_{\lbrace \vert P \vert > 1 \rbrace} = M. $ \\
We start by computing the derivative of the function $ \psi(t) = (1+t)^{\alpha} (\vert P(t) \vert -1) 1_{\lbrace \vert P \vert > 1 \rbrace} $ at the point $T$. We will prove later that it is strictly negative, showing that the bootstrap assumption can be continued past time $T.$ Note that the differentiation is legitimate by Corollary \ref{C1}.
\\
\\
Using the equation of motion of the particle, we get, using Corollary \ref{finalform}: 
\begin{align}
\label{ineq} \psi'(t) &= \frac{1+\alpha}{(3+2n)(1+ t)} \psi(t) \\
\notag & - \rho_0 \Lambda(\vert P(t) \vert) (1+t)^{-(1+\alpha) \frac{2n+2}{2n+3}} \psi(t)^{3+2n} \\
 \notag & + \frac{P(t)}{\vert P(t) \vert} \cdot (R_1+R_2+R_3+R_4) (1+t)^{\alpha} \\
 \notag &+   \frac{P(t)}{\vert P(t) \vert} \cdot (E_{1,\epsilon} + E_{2,\epsilon} + R_{\epsilon}) (1+t)^{\alpha}
\end{align}
This implies that
\begin{align}
\label{ineqbis} \psi'(T) &\leqslant  \frac{1+\alpha}{(3+2n)(1+ t)} \psi(T) \\
\notag & - \rho_0 \Lambda(\vert P(T) \vert) (1+T)^{-(1+\alpha) \frac{2n+2}{2n+3}} \psi(T)^{3+2n} \\
 \notag & + \Bigg \vert \frac{P(t)}{\vert P(t) \vert} \cdot (R_1+R_2+R_3+R_4) (1+T)^{\alpha} \Bigg \vert \\
 \notag &+ \Bigg \vert  \frac{P(t)}{\vert P(t) \vert} \cdot (E_{1,\epsilon} + E_{2,\epsilon} + R_{\epsilon}) (1+T)^{\alpha} \Bigg \vert
\end{align}
where the expressions of $R_1,R_2,R_3,R_4, E_{1,\epsilon}$ and $E_{2,\epsilon}$ are given in Lemma \ref{reform} and $R_{\epsilon}$ in Corollary \ref{finalform}.
\\
\\
We will prove that the terms involving $R_1,R_2,R_3$ and $R_4$ decay like $ t^{-1-\alpha} $ (as we mentioned previously we do not need to worry about $E_{1,\epsilon}, E_{2,\epsilon}$ and $R_{\epsilon}$ since we will take the limit as $\epsilon \to 0^+$ in the equation). 
\begin{remark} \label{alpha}
For $\alpha = \frac{1}{2n+2}$ the first three terms in \eqref{ineq} have the same decay in $t. $ 
\end{remark}
Then we will see that, using the smallness of the potential and the initial data of the field, the coupling term dominates over the other two terms in \eqref{ineq}. As a result $\psi'(T) <0$ and therefore the bootstrap assumption \eqref{bootstrap1} can be continued past $T.$ \\
Finally we will prove that, if the friction is weak enough (that is, if $\rho_0$ is small enough) then the first bootstrap assumption \eqref{bootstrap1} can also be continued past time $T.$ \\
We start by proving the announced decay of the remainder terms.

\subsection{Easier terms}
We start by proving decay estimates for the easier terms $R_1, 
E_{1,\varepsilon}, R_2, R_4$ and $E_{2,\varepsilon}.$ The classical stationary phase lemma \ref{classicdisp} from the appendix suffices to treat them.
\begin{lemma}\label{decayI}
Assume that the bootstrap assumptions \eqref{bootstrap1} and \eqref{bootstrap2} hold. \\
Then the following decay estimates hold:
\begin{align*}
\vert R_1 \vert & \lesssim \max \lbrace \vert P_0 \vert ;1 \rbrace \frac{\sqrt{\rho_0}\vert \vert \vert \beta_0 \vert \vert \vert }{(1+t)^{1+\alpha}}\Vert \nabla L W \Vert_{L^1}   \\
\vert E_{1,\epsilon} \vert & \lesssim \epsilon \frac{\sqrt{\rho_0} \vert \vert \vert \beta_0 \vert \vert \vert}{(1+t)^{1+\alpha}} \Vert \nabla W \Vert_{L^1}   \\
\vert R_2 \vert & \lesssim \frac{\sqrt{\rho_0} M^{3+2n} \vert \vert \vert \beta_0 \vert \vert \vert}{(1+t)^{1+\alpha}} \Vert \nabla W \Vert_{L^1}   \\
\vert R_4 \vert & \lesssim \rho_0\frac{\Vert \nabla W \Vert_{L^1} \big( \Vert W \Vert_{L^1} + \Vert U^{5/2} W \Vert_{L^1}\big)}{(1+t)^{1+\alpha}}  \\
\vert E_{2,\epsilon} \vert & \lesssim \epsilon t \rho_0 \Vert \nabla W \Vert_{L^1} \big( \Vert W \Vert_{L^1} + \Vert U^{5/2} W \Vert_{L^1}\big) 
\end{align*}
where the explicit expressions of $R_1, R_2, R_4, E_{1,\epsilon}$ and $E_{2,\epsilon}$ have been introduced in Lemma \ref{reform}. \\
Note also that the implicit constants here do not depend on $M$.
\end{lemma}
\begin{remark}
Note that the bounds on $R_1, R_2$ and $R_4$ are uniform on $\epsilon.$ 
\end{remark}
\begin{proof}
We have 
\begin{align*}
R_1 &= \frac{\sqrt{\rho_0}}{2(2\pi)^5} \Re \Big \langle i \xi_j \widehat{W} , \big( \widehat{\Re \beta_0} + \frac{\widehat{\Im \beta_0} \vert \xi \vert}{i \sqrt{1+\vert \xi \vert^2}} \big) \frac{(i\phi_1 (\xi) + i P(0) \cdot \xi )e^{it \phi_1 (\xi) + i (X(t)-X(0)) \cdot \xi}}{i\phi_1(\xi)+i P(t) \cdot \xi-\epsilon} \\
& + \big(\widehat{\Re \beta_0} - \frac{\widehat{\Im \beta_0} \vert \xi \vert}{i \sqrt{1+\vert \xi \vert^2}} \big) \frac{(-i\phi_1 (\xi) + i P(0) \cdot \xi )e^{-it \phi_1 (\xi) + i (X(t)-X(0)) \cdot \xi}}{-i\phi_1(\xi)+i P(t) \cdot \xi-\epsilon} \Big \rangle 
\end{align*}
We deal with the term
\begin{align*}
\frac{\sqrt{\rho_0}}{2(2\pi)^5} \Re \Big \langle i \xi_j \widehat{W} ,\widehat{\Re \beta_0}  \frac{i\phi_1(\xi) e^{it \phi_1 (\xi) + i (X(t)-X(0)) \cdot \xi}}{i\phi_1(\xi)+i P(t) \cdot \xi- \epsilon} \Big \rangle
\end{align*}
since the other parts are treated similarly. \\
First notice that 
\begin{align*}
\frac{1}{i\phi_1(\xi)+i P(t) \cdot \xi- \epsilon} = -\int_0 ^{\infty} e^{i\tau (\phi_1(\xi) + i P(t) \cdot \xi)} e^{-\epsilon \tau} d\tau
\end{align*}
Therefore
\begin{align}
&\label{termI} \frac{\sqrt{\rho_0}}{2(2\pi)^5} \Re \Big \langle i \xi_j \widehat{W}  ,\widehat{\Re \beta_0}  \frac{\phi_1(\xi) e^{it \phi_1 (\xi) + i (X(t)-X(0)) \cdot \xi}}{i\phi_1(\xi)+i P(t) \cdot \xi- \epsilon} \Big \rangle \\
&\notag = \frac{\sqrt{\rho_0}}{2(2\pi)^5} \Re \int_0 ^{\infty} \int_{\xi \in \mathbb{R}^5} i \vert \xi \vert \sqrt{1+\vert \xi \vert^2}  \xi_j \overline{ \widehat{W}} \widehat{\Re \beta_0} e^{i (t+\tau) \phi_1(\xi) + i(X(t)-X(0) \cdot \xi + i \tau P(t) \cdot \xi}e^{- \epsilon \tau} d\xi d\tau \\
&\notag = -\frac{\sqrt{\rho_0}}{2(2\pi)^5} \Im \int_0 ^{\infty} e^{-\epsilon \tau} \mathcal{F}^{-1}(\mathcal{F}(e^{i(t+\tau)L}((\partial_j L \widetilde{W}) \ast \Re \beta_0 )))(\tau P(t) + X(t)-X(0))  d\tau
\end{align}
where $\widetilde{W}(x) = W(-x).$ \\
Then using the standard stationary phase Lemma \ref{classicdisp} of the appendix we can write that
\begin{align*}
\vert \eqref{termI} \vert & \lesssim \sqrt{\rho_0}\Vert U^{3/2} \Re \beta_0 \Vert_{L^1} \Vert W \Vert' \int_0 ^{\infty} \frac{e^{-\epsilon \tau}}{(1+ t+\tau )^{5/2}} d\tau \\
 & \lesssim \frac{\sqrt{\rho_0}}{( 1+t )^{3/2}} \vert \vert \vert \beta_0 \vert \vert \vert \times \Vert W \Vert'
\end{align*}
which is the desired result. \\
All the other terms are handled in a very similar manner, therefore the proofs are omitted. 
\end{proof}

\subsection{The more challenging term}
There remains to treat the term $R_3,$ which is more involved than the terms seen above. Indeed an application of the stationary phase lemma along the lines of the previous subsection only provides a decay like $ \frac{1}{\sqrt{t}} .$  Proving that this term has an integrable decay rate requires a precise analysis of its phase. \\
We prove the following decay estimate on $\vert R_3 \vert:$
\begin{lemma}\label{decayIII}
Assume that the bootstrap assumptions \eqref{bootstrap1} and \eqref{bootstrap2} are satisfied. \\
Then we have
\begin{align*}
\vert R_3 \vert \leqslant \rho_0\frac{C(M)\Vert W \Vert'^2}{(1+t)^{1+\alpha}}
\end{align*}
for some constant $C(M)$ that depends on $M.$ 
\end{lemma}
The proof of this Lemma is given in Section \ref{challenging} below.

\subsection{Conclusion of the proof}
We start by taking the limsup as $\epsilon \to 0^+$ in \eqref{ineqbis} and obtain, given the bound on $E_{1,\epsilon}$ and $E_{2,\epsilon}$ from Lemma \ref{decayI} as well as Lemma \ref{representation}:
\begin{align*}
\psi'(t) &\leqslant \frac{1+\alpha}{(3+2n)(1+ t)} \psi(t) - \rho_0 \Lambda(\vert P(t) \vert) (1+t)^{-(1+\alpha) \frac{2n+2}{2n+3}} \psi(t)^{3+2n} \\
 \notag & + \bigg \vert \frac{P(t)}{\vert P(t) \vert} \cdot (R_1+R_2+R_3+R_4) (1+t)^{\alpha} \bigg \vert
\end{align*} 
This means that, as announced earlier, the terms $E_{1,\epsilon}, E_{2,\epsilon}$ and $R_{\epsilon}$ have no impact on the asymptotic behavior of the particle.\\
\\
Now by Lemma \ref{decayI} we see that if the initial data is chosen such that
\begin{align*}
\vert \vert \vert \beta_0 \vert \vert \vert & \leqslant C'\sqrt{\rho_0} \min \bigg \lbrace 1; \frac{1}{\vert P_0 \vert} \bigg \rbrace \\
\Vert W \Vert' &\leqslant C'
\end{align*}
for some small constant $C'$ (independent of all the parameters) then 
\begin{align*}
\big(\vert R_1 \vert + \vert R_2 \vert + \vert R_4 \vert \big)(1+T)^{\alpha} \leqslant \frac{\rho_0 C_0 M^{3+2n}}{100 (1+T)}
\end{align*}
Similarly we have 
\begin{align*}
C(M) \big(\Vert W \Vert'\big)^2  \leqslant \frac{C_0 M^{3+2n}}{100}
\end{align*}
and therefore 
\begin{align*}
(1+T)^{\alpha}\vert R_3 \vert \lesssim \rho_0 \frac{C(M) \big(\Vert W \Vert'\big)^2}{1+T} \leqslant \rho_0 \frac{C_0 M^{3+2n}}{100(1+T)}
\end{align*}
if $W$ is chosen such that
\begin{align*}
\Vert W \Vert' \leqslant C''(M)
\end{align*}
for $C''(M)$ a constant (that depends on $M$, therefore on $\rho_0$) that is small enough. \\
Then,
\begin{align*}
\psi'(T) \leqslant -\frac{49\rho_0 C_0 M^{3+2n}}{50(1+T)} + \frac{2M}{3(1+T)}
\end{align*}
But given the choice of $M,$ we have $\rho_0 C_0 M^2 > 50 M$ hence $\psi'(T) <0. $ \\
As a result the bootstrap assumption \eqref{bootstrap1} can be continued past time $T.$ 
\\
\\
Now we move on to the easier \eqref{bootstrap1}. \\
We know from Lemma \ref{representation} that $\Lambda$ is a smooth function on $(1;+\infty)$ and that $\vert P \vert$ is bounded by Lemma \ref{easybound}. Using these facts together with the previous smallness conditions on $W$ and the initial data for the field, we can write that 
\begin{align*}
\vert \dot{P}(t) \vert & \leqslant \rho_0 \big( \sup_{(1;\sup_t \vert P \vert]} \vert \Lambda \vert \big) (\vert P(t) \vert -1)^{2n+3} + \frac{\rho_0 C_0 M^{3+2n}}{50(1+T)^{1+\alpha}} \\
                       & \leqslant \frac{M^{3+2n}}{(1+T)^{1+\alpha}} \big( \rho_0  \sup_{(1;\sup_t \vert P \vert]} \vert \Lambda \vert + \frac{\rho_0 C_0}{50}  \big)
\end{align*}
which allows us to extend \eqref{bootstrap1} past time $T$ as long as $ \rho_0 $ is small enough. \\
This concludes the proof of the statement for the asymptotic behavior of the particle in theorem \ref{main}.
\end{proof}

\section{Proof of Lemma \ref{decayIII}} \label{challenging}
In this section we carry out the proof of Lemma \ref{decayIII}. \\
As mentioned above, a simple application of the stationary phase lemma gives a decay rate that is not integrable. To overcome this difficulty we localize precisely the stationary point of the phase, and find that, under our bootstrap assumptions \eqref{bootstrap1} and \eqref{bootstrap2}, it is very close to the origin. We can then exploit the Fermi Golden rule, which translates into the Fourier transform of the interaction potential cancelling at the origin, to gain decay.   
\subsection{Preparation of the proof}
Recall that
\begin{align*}
R_3 &= - \frac{\rho_0}{2(2\pi)^5} \Re \Big \langle i \xi_j \widehat{W}(\xi) ,  \frac{\widehat{W}(\xi) \vert \xi \vert}{i \sqrt{1+\vert \xi \vert^2}} \frac{e^{it \phi_1 (\xi) + i (X(t)-X(0)) \cdot \xi}}{i\phi_1(\xi)+i P(t) \cdot \xi-\epsilon} \\
&\times i \int_0 ^t \int_0 ^s \dot{P} (s) \cdot \xi e^{-i \tau \phi_1 (\xi) - i (X(\tau)-X(0)) \cdot \xi} ds d\tau \\
& - \frac{\widehat{W}(\xi) \vert \xi \vert}{i \sqrt{1+\vert \xi \vert^2}}  \frac{e^{-it \phi_1 (\xi) + i (X(t)-X(0)) \cdot \xi}}{-i\phi_1(\xi)+i P(t) \cdot \xi-\epsilon} i \int_0 ^t \int_0 ^s \dot{P} (s) \cdot \xi e^{i \tau \phi_1 (\xi) + i (X(\tau)-X(0)) \cdot \xi} ds d\tau \Big \rangle \\
&:= R_3a + R_3b
\end{align*}
We prove the result for $R_3a$. The second term $R_3b$ is handled in a similar fashion. We have that
\begin{align*}
R_3a &= - \frac{\rho_0}{2(2\pi)^5} \Re \Big \langle i \xi_j \widehat{W}, \frac{\widehat{W} \vert \xi \vert}{i \sqrt{1+\vert \xi \vert^2}} \frac{e^{it \phi_1 (\xi) + i (X(t)-X(0)) \cdot \xi}}{\phi_1(\xi)+i P(t) \cdot \xi-\epsilon} i \int_0 ^t \int_0 ^s \dot{P} (s) \cdot \xi e^{-i \tau \phi_1 (\xi) - i (X(\tau)-X(0)) \cdot \xi} ds d\tau \Big \rangle \\
& = \Im \frac{\rho_0}{2} \int_0 ^t \dot{P}_l (s) \int_0 ^s \int_0 ^{\infty} e^{- \sigma \epsilon} \int_{\mathbb{R}^5} \frac{\xi_j \xi_l \vert \xi \vert}{\sqrt{1+\vert \xi \vert^2}} \overline{\widehat{W}}(\xi) \widehat{W}(\xi) \\
& \times e^{i(t+\sigma -\tau) \phi_1 (\xi) + i(X(t)-X(\tau)) \cdot \xi + i \sigma P(t) \cdot \xi} d\xi d\tau d\sigma ds \\
& = \Im \frac{\rho_0}{2} \int_0 ^t \dot{P}_l (s) \int_0 ^s \int_0 ^{\infty} e^{- \sigma \epsilon} \int_{\mathbb{R}^5} \frac{\xi_j \xi_l \vert \xi \vert^{4n+1}}{\sqrt{1+\vert \xi \vert^2}} \vert \widehat{V} \vert ^2 (\xi) e^{i(t+\sigma -\tau) \phi_1 (\xi) + i(X(t)-X(\tau)) \cdot \xi + i \sigma P(t) \cdot \xi} d\xi d\tau d\sigma ds \\
& := \Im \frac{\rho_0}{2} \int_0 ^t \dot{P}_l (s) \int_0 ^s \int_0 ^{\infty} e^{- \sigma \epsilon} \mathcal{I}~ d\tau d\sigma ds, 
\end{align*}
where there is an implicit summation on $l$ in the expression above. \\
\\
We start by proving a bound on $\mathcal{I}.$ \\
Let $\phi:\mathbb{R} \rightarrow [0;1]$ be an even smooth function supported on $[-8/5;8/5]$ equal to 1 on $[-5/4;5/4].$ Let $\chi$ be defined as $\chi(x) = \phi(x)-\phi(2x).$ 
\\ Then write that
\begin{align}
\label{int}\vert \mathcal{I} \vert&\leqslant \sum_{k \in \mathbb{Z}} \Bigg \vert \int_{\mathbb{R}^5} \frac{\xi_j \xi_l \vert \xi \vert^{4n+1}}{\sqrt{1+\vert \xi \vert^2}} \vert \widehat{V} \vert ^2 (\xi) e^{i(t+\sigma -\tau) \phi_1 (\xi) + i(X(t)-X(\tau)) \cdot \xi + i \sigma P(t) \cdot \xi} \chi(2^{-k} \vert \xi \vert) d\xi \Bigg \vert  \\
\notag& =  \sum_{k \in \mathbb{Z}} 2^{k(8+4n)} \Bigg \vert \int_{\mathbb{R}^5} \frac{\xi_j \xi_l \vert \xi \vert^{4n+1}}{\sqrt{1+\vert 2^k \xi \vert^2}} \vert \widehat{V} \vert ^2 (2^k \xi) e^{i(t+\sigma -\tau) \phi_1 (2^k \xi) + i2^k (X(t)-X(\tau)) \cdot \xi + i \sigma P(t) \cdot \xi} \chi(\vert \xi \vert) d\xi \Bigg \vert
\end{align}
We focus on the inner integral in $\xi.$ 
\\
Now we seek to localize more precisely the stationary point in $\xi$ of the phase. Let's denote
\begin{align} \label{phase}
\Phi_1 (t,\tau,\sigma,k,\xi) &= (t+\sigma -\tau) \phi_1 (2^k \vert \xi \vert) + 2^k \xi \cdot \big( X(t)-X(\tau) + \sigma P(t) \big)
\end{align}
As we said, we will prove that the bootstrap assumptions imply that the stationary point is located close to the origin. 

\subsection{Localization of the stationary point}
We start with the following useful estimate:
\begin{lemma}\label{velocityapprox}
Assume that the bootstrap assumptions \eqref{bootstrap1} and \eqref{bootstrap2} are satisfied. \\
Then for every $ \sigma >0,0 \leqslant \tau \leqslant t, $ we can write
\begin{eqnarray*}
\frac{1}{t+\sigma - \tau} (\sigma P(t) + (X(t)-X(\tau))) &=:& P(t) + \delta(t)
\end{eqnarray*}
where
\begin{eqnarray*}
\vert \delta(t) \vert & \leqslant & \min \Bigg \lbrace \frac{199 M^{3+2n}}{(1+ t)^{\alpha}}; 4 M \Bigg \rbrace
\end{eqnarray*}
\end{lemma}
\begin{proof}
First we prove the second bound, which directly follows from Lemma \ref{easybound} and the mean-value theorem:
\begin{align*}
\vert \delta (t) \vert \leqslant 2\vert P(t) \vert + \frac{\vert X(t)-X(\tau) \vert}{t-\tau} \leqslant 4 M
\end{align*}
For the second estimate, we start with the following preliminary result: \\
Using the bootstrap assumption on the decay of the acceleration, we have:
\\
If $t \geqslant 1:$
\begin{eqnarray*}
\bigg \vert \frac{X(t)}{t} - P(t) \bigg \vert & = & \frac{1}{t} \bigg \vert \int_0 ^t P(s) ds - \int_0 ^t P(t) ds \bigg \vert \\
& = & \frac{1}{t} \bigg \vert \int_0 ^t \int_s ^t \dot{P}(\tau) d\tau ds \bigg \vert \\
& \leqslant & \frac{1}{t} \bigg \vert \int_0 ^t s \frac{M^{3+2n}}{(1+s)^{1+\alpha}} ds \bigg \vert \\ 
& \leqslant & \frac{10 M^{3+2n}}{(1+t)^{\alpha}}
\end{eqnarray*}
and if $t \leqslant 1:$
\begin{align*}
\bigg \vert \frac{X(t)}{t} - P(t) \bigg \vert &\leqslant \frac{1}{t} \int_0 ^t \vert P(s) - P(t) \vert ds \\
& \leqslant 2 \sup_{t \in [0;t]} \vert P(t) \vert \\
& \leqslant 2 M
\end{align*}
Overall the following estimate holds for times $t \in [0;T]:$
\begin{align*}
\bigg \vert \frac{X(t)}{t} - P(t) \bigg \vert \leqslant \frac{20 M^{3+2n}}{( 1+t )^{\alpha}}
\end{align*}
This result will be used in what follows. \\
We split the proof into two cases: \\
\\
\underline{Case 1: $ \tau > \frac{t}{2}$} \\
\begin{eqnarray*}
\sigma P(t) + (X(t)-X(\tau)) = (t+\sigma - \tau) P(t) + (X(t)-X(\tau)) - P(t)(t-\tau)
\end{eqnarray*}
and with the result above
\begin{align*}
\bigg \vert (X(t)-X(\tau)) - P(t) (t-\tau) \bigg \vert & = \bigg \vert \int_{\tau} ^{t} P(s) ds - \int_{\tau} ^{t} P(t) ds \bigg \vert \\
& \leqslant \int_{\tau} ^{t} \int_s ^t \vert \dot{P}(u) \vert du  ds \\
& \leqslant \int_{\tau} ^t \frac{s M^{3+2n}}{(1+s)^{\alpha}} ds \\
& \leqslant \frac{(t-\tau)M^{3+2n}}{(1+\tau)^{\alpha}} \\
& \leqslant \frac{5(t-\tau)M^{3+2n}}{(1+t)^{\alpha}}
\end{align*}
and the result follows after dividing by $t-\tau + \sigma.$ \\
\\
\underline{Case 2: $ \tau \leqslant \frac{t}{2}$} We write \\
\begin{eqnarray*}
\sigma P(t) + (X(t)-X(\tau)) &=& (t+\sigma - \tau) P(t) + (X(t)-t P(t)) \\
                             &-& (X(\tau)-\tau P(\tau)) + \tau(P(t)-P(\tau)) 
\end{eqnarray*}
and estimate the terms separately. For the first one, we use the preliminary result:
\begin{eqnarray*}
\frac{1}{t-\tau + \sigma}\vert X(t) - t P(t) \vert & = & 
\frac{t}{t-\tau + \sigma} \bigg \vert \frac{X(t)}{t}-P(t) \bigg \vert  \\
& \leqslant & \frac{20 t}{t-\tau + \sigma} \frac{M^{3+2n}}{(1+ t )^{\alpha}} \\
& \leqslant & \frac{40 M^{3+2n}}{(1+t) ^{\alpha}}
\end{eqnarray*}
For the second part, we write similarly that for $t \geqslant 1:$
\begin{eqnarray*}
\frac{1}{t-\tau+\sigma} \bigg \vert X(\tau) - \tau P(\tau) \bigg \vert & = & \frac{\tau}{t-\tau+\sigma} \bigg \vert \frac{X(\tau)}{\tau} -  P(\tau) \bigg \vert \\
&\leqslant & \frac{\tau}{t-\tau +\sigma} \frac{M^{3+2n}}{(1+ \tau )^{\alpha}} \\
& \leqslant & \frac{2 M^{3+2n} \tau^{1-\alpha}}{t} \\
& \leqslant & \frac{2 M^{3+2n}}{t^{\alpha}}
\end{eqnarray*}
and we also have using the mean value theorem that (estimate for $t$ small)
\begin{eqnarray*}
\frac{1}{t-\tau+\sigma} \bigg \vert X(\tau) - \tau P(\tau) \bigg \vert & = & \frac{\tau}{t-\tau+\sigma} \bigg \vert \frac{X(\tau)}{\tau} -  P(\tau) \bigg \vert \\
&\leqslant & \frac{2 \tau \sup_{[0;t]} \vert P(t) \vert}{t-\tau +\sigma}  \\
& \leqslant & \frac{2 M \tau}{t-\tau + \sigma} \\
& \leqslant & M^{3+2n}
\end{eqnarray*}
In conclusion
\begin{align*}
\frac{1}{t-\tau+\sigma} \bigg \vert X(\tau) - \tau P(\tau) \bigg \vert & \leqslant \frac{50 M^{3+2n}}{(1+ t)^{\alpha}} 
\end{align*}
For the last piece, we define $ I $ the (nonzero by assumption in this case 2) integer such that 
\begin{eqnarray*}
\frac{t}{2^{I+1}} < \tau \leqslant \frac{t}{2^I}
\end{eqnarray*} 
Then using the fundamental theorem of calculus, the bootstrap assumption on the norm of the acceleration and the above inequality, we get:
\begin{eqnarray*}
\frac{\tau}{t-\tau + \sigma} \vert P(t) - P(\tau) \vert & \leqslant & \frac{\tau}{t-\tau + \sigma} \bigg \vert \sum_{l=0} ^{I-1} P(2^{-l} t) - P(2^{-l-1}t) \bigg \vert + \frac{\tau}{t-\tau + \sigma} \vert P(2^{-I}t) - P(\tau) \vert \\
& \leqslant & \frac{2^{-I}t}{t/2} \sum_{l=0} ^{I-1} \frac{t}{2^l} \frac{M^{3+2n}2^{(1+\alpha)l}}{\langle t \rangle ^{1+\alpha}} +  \frac{\tau}{t-\tau + \sigma} (2^{-I}t- \tau) \frac{M^{3+2n}}{(1+ \tau )^{1+\alpha}} \\
& \leqslant & \frac{2^{1-I}M^{3+2n}}{( 1+t)^{\alpha}} \sum_{l=0}^{I-1} 2^{l \alpha} + \frac{1}{2} \frac{1}{t-\tau+\sigma} 2^{-I} t  \frac{M^{3+2n}}{(1+ \tau )^{\alpha}} \\
& \leqslant &  \frac{2^{1-I}M^{3+2n}}{(1+t)^{\alpha}} 2^{1+I \alpha} + \frac{\tau}{t} \frac{M^{3+2n}}{(1+ \tau)^{\alpha}} \\
& \leqslant & \frac{4 M^{3+2n}}{(1+t)^{\alpha}} +M^{3+2n} \min \bigg \lbrace 1; \frac{\tau^{1-\alpha}}{t} \bigg \rbrace \\
& \leqslant & \frac{20 M^{3+2n}}{(1+ t)^{\alpha}}
\end{eqnarray*}
\end{proof}
We use this estimate to localize the stationary point in the following lemma. It shows that stationary points in $\xi$ of the phase $\Phi_1$ are localized near the origin. 
\begin{lemma}\label{local}
Assume that the bootstrap assumptions \eqref{bootstrap1} and \eqref{bootstrap2} are satisfied. \\
Let $k_0$ be the smallest integer such that 
$$2^{k_0} \geqslant \frac{200M^{n+2}}{(1+ t )^{\alpha/2}}.
$$
Assume also that $k > k_0+10.$ \\
Then the phase $\Phi_1$, whose formula is given by \eqref{phase}, is not stationary. \\
Moreover we have the bound
\begin{align*}
\vert \nabla \Phi_1 \vert \gtrsim \min \lbrace 1; 2^k \rbrace (t+\sigma -\tau)
\end{align*}
\end{lemma}
\begin{proof}
Let's suppose that $\xi$ is a stationary point of the phase $\Phi_1 .$ \\
Then it must satisfy for $i=1,2,3:$
\begin{align*}
2^k (t+\sigma-\tau) \frac{\xi_i}{\vert \xi \vert} \bigg( \sqrt{1+2^{2k} \vert \xi \vert^2} + \frac{2^{2k} \vert \xi \vert^2}{\sqrt{1+2^{2k} \vert \xi \vert^2}}  \bigg) = 2^k \big( X(t)-X(\tau) + \sigma P(t) \big)_i
\end{align*}
where $Y_i$ denotes the $i-$th component of the vector $Y.$ \\
Now we multiply the $i-th$ equation by $\xi_i,$ sum on $i$ and use the Cauchy-Schwarz inequality to obtain
\begin{align*}
\vert \xi \vert \big( \sqrt{1+2^{2k} \vert \xi \vert^2} + \frac{2^{2k} \vert \xi \vert^2}{\sqrt{1+2^{2k} \vert \xi \vert^2}} \big)(t+\sigma-\tau) \leqslant \vert  X(t)-X(\tau) + \sigma P(t) \vert \vert \xi \vert
\end{align*}
This implies, using Lemma \ref{velocityapprox}, that
\begin{align*}
1+2^{2k-2} \leqslant (\vert P(t) \vert + \vert \delta(t) \vert)^2
\end{align*}
which, using Lemma \ref{velocityapprox} again, gives us
\begin{align*}
2^{2k-2} &\leqslant (\vert P(t) \vert + \vert \delta(t) \vert -1)(1+\vert P(t) \vert + \vert \delta(t) \vert) \\
         &\leqslant \frac{1000 M^{4+2n}}{(1+t)^{\alpha}}
\end{align*}
and the results follow directly.

\end{proof}

\subsection{End of the proof of Lemma \ref{decayIII}}
The previous computation shows that there are two natural ranges of the index $k$ to consider. 
\\
Let $k_0$ be the smallest integer such that 
$$2^{k_0} > \frac{200 M^{2+n}}{(1+ t)^{\alpha/2}}.
$$ 
First we consider the region where the phase is possibly stationary in $\xi$: \\
\\
\underline{Region 1: $ k \leqslant k_0 +10.$} \\
\\
Then we bound $\mathcal{I}$ using Young's inequality and write that
\begin{align*}
\vert \mathcal{I} \vert & \leqslant \sum_{k \leqslant k_0+10} \Bigg \vert  \int_{\mathbb{R}^5} \frac{\xi_j \xi_l \vert \xi \vert^{4n+1}}{\sqrt{1+\vert \xi \vert^2}} \vert \widehat{V} \vert ^2 (\xi) e^{i(t+\sigma -\tau) \phi_1 (\xi) + i(X(t)-X(\tau)) \cdot \xi + i \sigma P(t) \cdot \xi} \chi(2^{-k} \vert \xi \vert) d\xi \Bigg \vert \\
\notag& = \sum_{k \leqslant k_0+10} \Bigg \vert \mathcal{F}^{-1} \Bigg( e^{i(t + \sigma - \tau) \phi_1 (\xi) } \vert \xi \vert^{4n+3} \chi(2^{-k} \vert \xi \vert) \frac{\xi_i \xi_l}{\vert \xi \vert^2 \sqrt{1+\vert \xi \vert^2}} \vert \widehat{V} \vert^2 (\xi) \Bigg) (X(t)-X(\tau) - \sigma P(t)) \Bigg \vert \\
& \lesssim \sum_{k \leqslant k_0+10} \Bigg \Vert \mathcal{F}^{-1} \bigg(  e^{i(t + \sigma - \tau) \phi_1 (\xi)} \vert \xi \vert^{4n+3} \chi(2^{-k} \vert \xi \vert) \bigg) \Bigg \Vert_{L^{\infty}} \Bigg \Vert \mathcal{F}^{-1} \bigg( \frac{\xi_i \xi_l}{\vert \xi \vert^2} \frac{1}{\sqrt{1+\vert \xi \vert^2}} \tilde{\chi}(2^{-k}\xi) \vert \widehat{V} \vert^2 \bigg) \Bigg \Vert_{L^1} \\
& \lesssim \sum_{k \leqslant k_0+10} \Bigg \Vert \mathcal{F}^{-1} \bigg(  e^{i(t + \sigma - \tau) \phi_1 (\xi)} \vert \xi \vert^{4n+3} \chi(2^{-k} \vert \xi \vert) \bigg) \Bigg \Vert_{L^{\infty}} \Vert V \Vert_{L^1} \bigg \Vert  \frac{R_j R_l}{\sqrt{1-\Delta}} V \bigg \Vert_{L^1}
\end{align*}
where $\tilde{\chi}$ denotes a slightly enlarged version of $\chi.$
Now we focus on estimating the $L^{\infty}$ norm. 
\\
This term is bounded using the usual stationary phase lemma. For completeness we write the details of the proof since we have to keep track of the cancellations at 0. More precisely we have to pay attention not only to the decay in time but to the $2^k$ factors as well. 
\\
\\
Using the formula of the inverse Fourier transform of a radial function, we find that
\begin{align} 
\notag J_k &:= \mathcal{F}^{-1} \bigg(  e^{i(t + \sigma - \tau) \phi_1 (\xi)} \vert \xi \vert^{4n+3} \chi(2^{-k} \vert \xi \vert) \bigg)(x) \\
\label{express}&= \int_0 ^{\infty} e^{i(t+\sigma - \tau)\phi_1 (r)} K(r \vert x \vert)  \chi(2^{-k} r) r^{4n+7} dr 
\end{align}
where $K$ denotes the Fourier transform of the surface measure of the sphere. Moreover, by Lemma \ref{Fourier:sphere} of the appendix, we can write 
\begin{align*}
K(r) =  \frac{e^{ir} k_1(r) + e^{-ir} k_2(r)}{2}
\end{align*}
where $k_1$ and $k_2$ have favorable decay properties stated in this same lemma.
\\
Plugging this information in \eqref{express} we find that
\begin{align*}
J_k &= \frac{1}{2} \int_0 ^{\infty} e^{i(t+\sigma - \tau)\phi_1 (r)} e^{-ir \vert x \vert} k_1 (r \vert x \vert) \chi(2^{-k} r) r^{4n+7} dr \\
&+  \frac{1}{2} \int_0 ^{\infty} e^{i(t+\sigma - \tau)\phi_1 (r)} e^{ir \vert x \vert} k_2 (r \vert x \vert) \chi(2^{-k} r) r^{4n+7} dr \\
&= 2^{k-1} \int_0 ^{\infty} e^{i(t+\sigma - \tau)\phi_1 (2^k r)} e^{-i 2^k r \vert x \vert} k_1 (2^k r \vert x \vert) \chi(r) 2^{(4n+7)k} r^{4n+7} dr \\
&+  2^{k-1} \int_0 ^{\infty} e^{i(t+\sigma - \tau)\phi_1 (2^k r)} e^{i 2^k r \vert x \vert} k_2 (2^k r \vert x \vert) \chi(r) 2^{(4n+7)k} r^{4n+7} dr \\
&:=J_k ^1  + J_k ^2
\end{align*} 
Decay estimates for these two pieces are given in the following lemmas. 
\begin{lemma}[The non stationary region in $r$] \label{decay1}
Assume that $ 2^k (t+\sigma -\tau) \phi_1 ' (2^k r) \ll 2^k  \vert x \vert  $ or $2^k  \vert x \vert \ll 2^k (t+\sigma -\tau) \phi_1 ' (2^k r).$
\\
Then for any integer $N$ we have
\begin{align*}
\vert J_k ^1 \vert \lesssim \frac{1}{(1+t-\tau +\sigma)^N} 2^{k(4n+8-N)}
\end{align*}
\end{lemma}
\begin{proof}
This is a classical non stationary phase lemma. \\
We assume that $t-\tau + \sigma \geqslant 1,$ since the proof in the other case $0 \leqslant t- \tau + \sigma \leqslant 1$ is easier. \\ 
We have 
$$
\phi_1'(r) =\sqrt{1+r^2} + \frac{r^2}{\sqrt{1+r^2}}
$$
therefore we directy get that 
\begin{align}
\label{phase1:highder} 
\frac{d^N}{dr^N} \bigg( \frac{1}{\phi_1'(2^k r)}  \bigg) \lesssim_N 1
\end{align}
Now we integrate by parts to prove decay. \\ Let $L$ be the  operator $L=\frac{1}{i \phi_1'(r)} \frac{d}{dr}. $ We denote $L_{\star}$ its adjoint.\\
We write 
\begin{align*}
J_k ^1 &= 2^{k(4n+8)} \int_0 ^{\infty} L^N \bigg(e^{i\phi_1(r)}\bigg) k_1 (2^k r \vert x \vert) \chi(r)  r^{4n+7} dr \\
&= \frac{2^{k(4n+8)}}{\big(i(t-\sigma+\tau) 2^k\big)^N} \int_0 ^{\infty} e^{i\phi_1(r)}L_{\star}^N \bigg( k_1 (2^k r \vert x \vert) \chi(r)r^{4n+7} \bigg) dr 
\end{align*}
Given Lemma \ref{Fourier:sphere} and \eqref{phase1:highder} we have 
\begin{align*}
\bigg \vert L_{\star}^N \bigg( k_1 (2^k r \vert x \vert) \chi(r)r^{4n+7} \bigg) \bigg \vert \lesssim_{N} 1
\end{align*}
and the result follows.
\end{proof}
We also have a similar estimate for $J_k ^2$ since in this case the phase is non-stationary:
\\
The proof is omitted given the similarity with the previous case. 
\begin{lemma}\label{decay1prime}
We have the bound for any integer $N:$
\begin{align*}
\vert J_k ^2 \vert \lesssim_N \frac{2^{k(4n+8-N)}}{(1+t-\tau+\sigma)^N}
\end{align*}
\end{lemma}
For the remaining case we use the stationary phase lemma: 
\begin{lemma}[The stationary region in $r$]\label{decay2}
Assume that $2^k  \vert x \vert \simeq 2^k (t+\sigma -\tau) \phi_1 ' (2^k r).$ Then we have the bound
\begin{align*}
\vert J_k ^1 \vert \lesssim \frac{2^{k(4n+5)}}{(1+t-\tau+\sigma)^{5/2}}
\end{align*}
\end{lemma}
\begin{proof}
We assume that $t+\tau-\sigma \geqslant 1$ since the proof in the other case is easier. \\
First a computation gives 
$$
\phi_1''(r) = \frac{r}{\sqrt{1+r^2}} + \frac{1}{(1+r^2)^{\frac{3}{2}}}
$$
Hence
\begin{align*}
\vert \phi_1''(r) \vert \gtrsim (t+\sigma - \tau) 2^{2k} 
\end{align*}
Therefore we find, using Van der Corput's lemma (see for example \cite{St}, chapter VIII, proposition 2) as well as the decay properties of $k_1,k_2$ of Lemma \ref{Fourier:sphere} to obtain
\begin{align*}
\vert J_k ^1 \vert & \lesssim \frac{1}{2^k (t+\sigma - \tau)^{1/2}} \int_0 ^{\infty} \bigg \vert \frac{d}{dr} \bigg( k_1 (2^k r \vert (X(t)-X(\tau)) + \sigma P(t) \vert) \chi(r) 2^{(4n+7)k} r^{4n+7} \bigg) \bigg \vert dr \\ 
& \lesssim \frac{1}{2^k (t+\sigma - \tau)^{1/2}} \frac{2^{k(4n+8)}}{(1+2^k \vert x \vert)^{2}}
\\
& \lesssim \frac{2^{k(4n+5)}}{(t+\sigma - \tau)^{5/2}} 
\end{align*}
which is the desired result
\end{proof}
Now we treat the case of the second region that is not stationary in $\xi$: 
\\
\\
\underline{Region 2: $k>k_0+10$}
\\
\\
In this case, by Lemma \ref{local}, the phase is not stationary. Since we do not want to destroy the localization of the phase, we do not use Young's inequality but integrate by parts directly in the following integral:
\begin{align*}
P_k:=\int_{\mathbb{R}^5} \frac{\xi_j \xi_l \vert \xi \vert^{4n+1}}{\sqrt{1+\vert 2^k \xi \vert^2}} \vert \widehat{V} \vert ^2 (2^k \xi) e^{i(t+\sigma -\tau) \phi_1 (2^k \xi) + i2^k (X(t)-X(\tau)) \cdot \xi + i \sigma P(t) \cdot \xi} \chi(\vert \xi \vert) d\xi
\end{align*}
More precisely we have the following lemma:
\begin{lemma}\label{decay3}
We have the bound 
\begin{align*}
\vert P_k \vert \lesssim_{N} \frac{2^{k(4n+3-N)} + 2^{k(4n+3)} }{(1+t+\tau-\sigma)^N}  \big( \Vert \chi(2^{-k}\vert \xi \vert) \vert \hat{V} \vert^2 \Vert_{L^1} + \Vert \chi(2^{-k}\vert \xi \vert) \nabla^N \big( \vert \hat{V} \vert^2 \big) \Vert_{L^1} \big)
\end{align*}
\end{lemma}
\begin{proof}[Sketch of the proof]
The proof is very similar to Lemma \ref{decay1}. Here we do not switch to polar coordinates, but instead use the operator 
$$
\tilde{L} = \frac{\nabla \Phi_1}{\vert \nabla \Phi_1 \vert} \cdot \nabla  
$$
We skip the details since the proof is classical and very close to that of Lemma \ref{decay1}.
\end{proof}
With these decay estimates we can finish the proof of Lemma \ref{decayIII}:
\begin{proof}[Proof of Lemma \ref{decayIII}]
We split the estimates into three parts:
\begin{align*}
\vert IIIa \vert \lesssim \rho_0 \int_0 ^t \dot{P}_l (s) \int_0^s \int_0^{\infty} e^{-\sigma \epsilon} \bigg( \sum_{k > k_0} \vert P_k \vert+ \sum_{k \leqslant k_0} \vert J_k ^1 \vert + \sum_{k \leqslant k_0} \vert J_k ^2 \vert \bigg) d\tau d\sigma ds
\end{align*}
For the first part we use Lemma \ref{decay3} for $N=4$ as well as the bootstrap assumption \eqref{bootstrap1}:
\begin{align*}
&\rho_0 \int_0 ^t \dot{P}_l (s) \int_0^s \int_0^{\infty} e^{-\sigma \epsilon} \sum_{k > k_0} \vert P_k \vert d\tau d\sigma ds \\
& \lesssim \rho_0 \sum_{k > k_0} (2^{k(4n-1)} + 2^{k(4n+3)} ) \int_0 ^t \frac{M^{3+2n}}{(1+ s)^{1+\alpha}} \int_0 ^s \int_0 ^{\infty} \frac{1}{(1+ t-\tau + \sigma)^{4}} d\sigma d\tau ds \\
& \times \big( \Vert \chi(2^{-k}\vert \xi \vert) \vert \hat{V} \vert^2 \Vert_{L^1} + \Vert \chi(2^{-k}\vert \xi \vert) \nabla^4 \big( \vert \hat{V} \vert^2 \big) \Vert_{L^1} \big) \\
& \lesssim \rho_0 \big(\Vert W \Vert' \big)^2 \int_0 ^t \frac{M^{3+2n}}{(1+ s)^{1+\alpha}} \int_0 ^s \frac{1}{(1+ t-\tau)^{3}}  d\tau ds  \\
& \lesssim \rho_0 \big(\Vert W \Vert' \big)^2 \int_0 ^t \frac{M^{3+2n}}{(1+ s)^{1+\alpha}} \frac{1}{(1+ t-s)^{2}} ds 
\end{align*}
and the result follows.
\\
For the term involving $J_k^2$ we can use a similar reasoning to get the following:
\begin{align*}
&\rho_0 \int_0 ^t \dot{P}_l (s) \int_0^s \int_0^{\infty} e^{-\sigma \epsilon} \sum_{k \leqslant k_0} \vert J_k ^2 \vert  d\tau d\sigma ds \\
&\lesssim \rho_0 \int_0 ^t \frac{M^{3+2n}}{(1+ s)^{1+\alpha}} \frac{1}{(1+ t-s)^{2}} ds \bigg \Vert \frac{R_j R_l}{\sqrt{1-\Delta}} V \bigg \Vert_{L^1} \Vert V \Vert_{B^{4n+3}_{1,1}} \\
&\lesssim \rho_0 \int_0 ^t \frac{M^{3+2n}}{(1+ s)^{1+\alpha}} \frac{1}{(1+ t-s)^{2}} ds \big( \Vert W \Vert' \big)^2
\end{align*}
\begin{remark}\label{remarquefermi}
This is at this point that we use the Fermi golden rule (vanishing of $\widehat{W}(0)$). Note that without major changes to the proof (namely choosing $N=3+\frac{n}{100}$) one could allow $n$ to be any strictly positive real number. 
\end{remark}
For the last part we use Lemma \ref{decay2} as well as the bootstrap assumption \eqref{bootstrap1}:
\begin{align*}
&\rho_0 \int_0 ^t \dot{P}_l (s) \int_0^s \int_0^{\infty} e^{-\sigma \epsilon} \sum_{k \leqslant k_0} \vert J_k ^1 \vert  d\tau d\sigma ds \\
&\lesssim \rho_0  \sum_{k \leqslant k_0}  2^{k(4n+5)} \int_0 ^t \frac{M^{3+2n}}{(1+ s )^{1+\alpha}} \int_0 ^s \int_0 ^{\infty} \frac{d\tau d\sigma ds}{(1+ t-\tau+\sigma)^{5/2}} \bigg \Vert \frac{R_j R_l}{\sqrt{1-\Delta}} V \bigg \Vert_{L^1} \Vert V \Vert_{L^1} \\
& \lesssim \rho_0  2^{(4n+5)k_0} \int_0 ^s \frac{M^{3+2n}}{(1+ s )^{1+\alpha}} \frac{ds}{(1+ t-s) ^{1/2}} \bigg \Vert \frac{R_j R_l}{\sqrt{1-\Delta}} V \bigg \Vert_{L^1}  \Vert V \Vert_{L^1}\\
& \lesssim \rho_0  M^{(n+2)(4n+5)} (1+t)^{-\frac{4n+5}{4n+4}} \frac{M^{3+2n}}{\sqrt{1+t}} \bigg \Vert \frac{R_j R_l}{\sqrt{1-\Delta}} V \bigg \Vert_{L^1} \Vert V \Vert_{L^1} \\
& \lesssim \rho_0  M^{(n+2)(4n+5)} (1+t)^{-\frac{4n+5}{4n+4}} \frac{M^{3+2n}}{\sqrt{1+t}} \bigg \Vert \frac{R_j R_l}{\sqrt{1-\Delta}} V \bigg \Vert_{L^1} \Vert V \Vert_{L^1}
\end{align*}
The result follows from these estimates.
\end{proof}



\section{Scattering of the field}\label{scattering}
Here we prove the scattering statement for the field. First let's define the function $S$ which solves
\begin{eqnarray*}
0 &=& H(\infty) S - \rho_0 \left( \begin{array}{ll}
0 \\
W
\end{array} \right) 
\end{eqnarray*}
Recall that the field solves
\begin{eqnarray*}
\partial_t h &=& H(t) h - \rho_0 \left( \begin{array}{ll}
0 \\
W
\end{array} \right) 
\end{eqnarray*}
Hence the difference $ \psi = h - S $ solves
\begin{eqnarray*}
\partial_t \psi &=& H(t) \psi + (H(t) - H(\infty)) S
\end{eqnarray*}
We write Duhamel's formula for this difference field and obtain:
\begin{eqnarray*}
\psi(t) &=& e^{R(t)} \psi_0 +  \int_0 ^{t} (P(s)-P_{\infty}) \cdot \nabla e^{R(t)-R(s)} S ds 
\end{eqnarray*}
Now using classical decay estimates for the semi-group $e^{itL},$ the diagonalization Lemma \ref{diag}, and the decay of the acceleration proved in Proposition \ref{mainprop} we find that 
\begin{align*}
\Vert \psi(t) \Vert_{L^{\infty}} &\leqslant \frac{\vert  \vert \vert \beta_0 \vert \vert \vert}{(1+t)^{5/2}} + \int_0 ^t \frac{1}{(1+s)^{\alpha}} \frac{ds}{(1+t-s)^{5/2}} \vert \vert \vert \nabla S \vert \vert \vert
\end{align*}
This directly implies that
\begin{align*}
\lim_{t \to \infty} \Vert \psi(t) \Vert_{L^{\infty}} =0
\end{align*}
which is the desired result.

\appendix
\section{Some elementary Lemmas}
We have the following Lemma, whose proof is a simple linear algebra calculation:
\begin{lemma}[Diagonalization of $H(t)$]\label{diag}
Let $A$ be the matrix
\begin{displaymath} A:= \left( \begin{array}{ll}
\sqrt{-\Delta} & \sqrt{-\Delta} \\
i \sqrt{1-\Delta} &- i \sqrt{1-\Delta}
\end{array} \right)
\end{displaymath}
Then 
\begin{displaymath} A^{-1}:= \left( \begin{array}{ll}
\frac{1}{2\sqrt{-\Delta}} & \frac{1}{2 i\sqrt{1-\Delta}} \\
\frac{1}{2\sqrt{-\Delta}} & -\frac{1}{2 i\sqrt{1-\Delta}}
\end{array} \right)
\end{displaymath}
With these notations we have
\begin{displaymath}
A^{-1} H(t) A  = \left( \begin{array}{ll}
iL + P \cdot \nabla & ~~~~~~~~ 0 \\
~~~~~~~~0 & -iL + P \cdot \nabla
\end{array} \right)
\end{displaymath}
\end{lemma}
We also recall here the classical decay estimate on the semi-group $e^{itL}$ that is a special case of theorem 2.1 from \cite{Gus}: 
\begin{lemma}\label{classicdisp}
We have the following decay estimate:
\begin{displaymath}
\Vert e^{itL} f \Vert_{\infty} \lesssim t^{-\frac{5}{2}} \Vert U^{3/2} f \Vert_{L^1} 
\end{displaymath}
\end{lemma}
Finally we give the following well-known consequence of the stationary phase lemma for the decay of the Fourier transform of the surface measure of the sphere:
\begin{lemma}\label{Fourier:sphere}
We have the following 
\begin{align*}
\widehat{\sigma_{\mathbb{S}^{d-1}}}(x) = e^{i \vert x \vert} k_1 (\vert x \vert) + e^{-i \vert x \vert} k_2 (\vert x \vert)
\end{align*}
where $ \sigma_{\mathbb{S}^{d-1}}$ denotes the surface measure of the unit sphere in $\mathbb{R}^d$ and $k_1, k_2$ are two smooth functions such that for every integer $N \geqslant 0$ we have
\begin{align*}
\frac{d^N}{dr^N} k_i (r) \lesssim_N \langle r \rangle^{-\frac{d-1}{2}-N}
\end{align*}

\end{lemma}

\end{document}